\documentclass{amsart}
\usepackage[latin1]{inputenc}
\usepackage{amssymb}
\usepackage{amsmath}
\usepackage{amsbsy}
\usepackage{enumerate}
\usepackage{epsfig,color,caption}
\usepackage[dvipsnames]{xcolor}
\usepackage{colonequals}
\usepackage{makecell}
\usepackage{color}
\usepackage[pdftex,colorlinks,citecolor=blue]{hyperref}
\usepackage{cleveref}
\usepackage{tikz}
\usetikzlibrary{arrows}
\usetikzlibrary{positioning}
\usetikzlibrary{cd}
\usetikzlibrary{calc}
\usetikzlibrary{shapes}
\usepackage{comment}
\usepackage{longtable,graphics,multirow,ulem}
\usepackage{changepage}
\usetikzlibrary{matrix}
\usepackage{pgf,tikz}
\usepackage[all]{xy}

\newcounter{nootje}
\newtheorem{theorem}{Theorem}[section]
\newtheorem{lemma}[theorem]{Lemma}
\newtheorem{proposition}[theorem]{Proposition}
\newtheorem{corollary}[theorem]{Corollary}

\newtheorem{definition}[theorem]{Definition}
\newtheorem{example}[theorem]{Example}

\newtheorem{notation}[theorem]{Notation}
\newtheorem{rem}[theorem]{Remark}

\numberwithin{equation}{section}
\newcommand{\rk}{\mbox{rank }}

\newcommand{\ra}{\rightarrow}

\newcommand{\Gal}{\mbox{Gal}}

\newcommand{\PP}{ \mathbb{P}}

\newcommand{\Z}{\mathbb{Z}}

\newcommand{\N}{\mathbb{N}}

\newcommand{\Aut}{\mathrm{Aut}}

\makeatletter
\def\blfootnote{\xdef\@thefnmark{}\@footnotetext}

\title[Fields of definition of Elliptic fibrations]{Fields of definition of elliptic fibrations on covers of certain extremal rational elliptic surfaces\footnote{ L\MakeLowercase{ast} U\MakeLowercase{pdated:} \today}}

\author{
Victoria Cantoral-Farf\'an
}
\address[Victoria Cantoral-Farf\'an]{
\begin{itemize}
\item[-] KU Leuven, Department of Mathematics\\
Celestijnenlaan 200B, B-3001 Leuven, Belgium
\end{itemize}
} \email[Victoria Cantoral-Farf\'an]{victoria.cantoralfarfan@kuleuven.be}

\author{
Alice Garbagnati
}
\address[Alice Garbagnati]{
\begin{itemize}
\item[-]Universit\`a Statale degli Studi di Milano\, Dipartimento di Matematica\\
via Saldini, 50 I20133 Milano
\end{itemize}
} \email[Alice Garbagnati]{alice.garbagnati@unimi.it}

\author{
Cec\'ilia Salgado
}
\address[Cec\'ilia Salgado]{
\begin{itemize}
\item[-] Universidade Federal do Rio de Janeiro (UFRJ)\, Instituto de Matem\'atica\\ Cidade Universit\'aria, Ilha do Fund\~ao, Rio de Janeiro.
\end{itemize}
} \email[Cec\'ilia Salgado]{salgado@im.ufrj.br}

\author{
Antonela Trbovi\'c 
}
\address[Antonela Trbovi\'c]{
\begin{itemize}
\item[-] University of Zagreb, Department of Mathematics\\ Bijeni\v{c}ka cesta 30, 10000 Zagreb, Croatia
\end{itemize}
} \email[Antonela Trbovi\'c]{antonela.trbovic@math.hr}

\author{
Rosa Winter
}
\address[Rosa Winter]{
\begin{itemize}
\item[-] Universiteit Leiden, Mathematisch Instituut\\
Niels Bohrweg 1, 2333 CA Leiden, The Netherlands
\end{itemize}
} \email[Rosa Winter]{r.l.winter@math.leidenuniv.nl}

\begin{document}
	
\begin{abstract}
We study K3 surfaces over a number field $k$ which are double covers of extremal rational elliptic surfaces. We provide a list of all elliptic fibrations on certain K3 surfaces together with the degree of a field extension over which each genus one fibration is defined and admits a section. We show that the latter depends, in general, on the action of the cover involution on the fibers of the genus 1 fibration.
\end{abstract}


	\maketitle
	\markleft{}

	\section{Introduction}
	
One main distinction of K3 surfaces, among others, is that they form the only class of surfaces that might admit more than one elliptic fibration with section, which is not of product type \cite[Lemma 12.18]{SS}. It is therefore a natural problem to classify such fibrations. This has been done in the past three decades, via different methods by several authors, see for instance \cite{Oguiso}, \cite{Nish}, \cite{Kloos}, \cite{Wine1}, \cite{CG},\cite{GS2} and \cite{Wine2}. Recently, the second and third authors have proposed a new method to classify elliptic fibrations on K3 surfaces which arise as double cover of rational elliptic surfaces. We refer the reader to \cite{GS} and \cite {GS2} for more details.

Let $X$ be a K3 surface obtained as a double cover of an extremal rational elliptic surface defined over a number field $k$. The purpose of this paper is to determine fields of definition of the distinct elliptic fibrations on $X$, i.e., fields over which the classes of the fiber and of at least one section are defined (see Def. \ref{def: fields of def}). We also determine, in some examples, an upper bound for the degree of the field over which the Mordell--Weil group admits a set of generators. Extremal rational elliptic surfaces have been classified by Miranda and Persson in \cite{MiPe}. There are sixteen configurations of singular fibers on such surfaces.
We restrict further our attention to smooth double covers of extremal rational elliptic surfaces with distinct reducible fibers, i.e. such that there are no two reducible fibers of the same Kodaira type. Given a genus 1 fibration on such a K3 surface, we show that it admits a section over a field that depends on the action of the cover involution on its fibers (see Theorem \ref{theo: types and field of definition}). We illustrate this last result for K3 surfaces that arise as a double cover branched over two smooth fibers of the extremal rational elliptic surfaces with one unique reducible fiber and also on smooth double covers of the surface with fiber configurations either $(III^*,I_2)$ or $(III^*,III)$. Remark that among those sixteen configurations of singular fibers on extremal rational elliptic surfaces only four of them have a unique reducible fiber, namely $(I_9, 3I_1)$, $(II^*, II), (II^*,2I_1)$ and $(I_4^*, 2I_1)$. As only the configuration of reducible fibers plays a role in our arguments, we narrow these down to three classes and study those extremal rational elliptic surfaces, denoted by $R_9$, $R_2$, and $R_4$ and the corresponding K3 surfaces $X_9$, $X_2$, and $X_4$, respectively. We denote by $R_3$ an extremal rational elliptic surface with fibers either $(III^*,I_2)$ or $(III^*,III)$ and its generic K3 cover $X_3$. Notice that the surface $X_4$ also occurs as a double cover of $R_3$ and hence, $X_3$ and $X_4$ belong to the same family of K3 surfaces. A reason to explore elliptic fibrations on $X_i$, $i=2,3,4,9$ is that they have different behavior with respect to the cover involution of $X_i\ra R_i$. Fibrations that are preserved by this involution are easier to describe via linear systems of curves on a rational surface, and one can also exhibit a Weierstrass equation for those as pointed out in \cite{Wine2} and \cite{GS2}. In particular, on $X_3$ and $X_4$, which can be identified, we have two different involutions (induced by the covers $X_4\ra R_4$ and $X_3\ra R_3$) and the behavior of each fibration on $X_3\simeq X_4$ with respect to these two involutions can be different. 

This paper is organized as follows. In Section \ref{sec: setting} we introduce the notations which remain in force during the paper and lay down our setting. Section \ref{sec: rational curves on K3 surfaces} is devoted to the study of rational curves on the K3 surface $X$ obtained as a double cover of a rational elliptic surface $R$. More precisely, motivated by the work done in \cite{GS} and \cite{GS2}, we study the behavior of the image by the quotient map~$\pi:X\to R$~of rational curves on $X$ and we determine the rational curves on $X$ coming from a section defined over $k$ of the elliptic fibration $\mathcal{E}_R$. While Section  \ref{sec: rational curves on K3 surfaces} is of geometric nature, Section \ref{sec: extremal RES} is dedicated to study the arithmetic of extremal rational elliptic surfaces defined over $k$. In particular,  we obtain the quite intriguing fact that with a possible unique exception all extremal rational elliptic surfaces can be obtained, over the ground field, as a blow-up of base points of a pencil of genus one curves in $\mathbb{P}^2$ or $\mathbb{P}^1\times\mathbb{P}^1$,  Lemma \ref{models over QQ}. Section \ref{sec: double covers on ext. RES} is dedicated to the study of K3 surfaces coming from double covers of extremal rational elliptic surfaces. We prove in Theorem \ref{theo: types and field of definition} that a genus 1 fibration on $X$ admits a section over a field which depends on the action of the cover involution on the fibers of the genus 1 fibration. Finally, in Sections \ref{sec: surfaces R9 and X9} and \ref{sec: surfaces Ri and Xi} we illustrate the previous result. More precisely, in Section 6 we give a classification of elliptic fibrations on the surface $X_9$ given by a generic double cover of an extremal rational elliptic surface $R_9$ with an $I_9$. We present a fiber class corresponding to each fibration on $X_9$ using sections and components of the reducible fibers of the fibration induced by the elliptic fibration on $R_9$. We also study the Mordell--Weil groups of each fibration and the fields of definition of the fibrations and their Mordell--Weil groups. Section 7 has similar results for the K3 covers of the rational elliptic surfaces $R_2,R_3$ and $R_4$, with reducible fibers $(II^*),(III^*,I_2)$ and $(I_{4}^*)$, respectively.

\subsection{Relation to the literature}
Fields of definition of the Mordell--Weil group of non-isotrivial elliptic surfaces were studied independently by Swinnerton--Dyer in \cite{SD} and Kuwata in \cite{Kuwata} via different methods than the ones presented here. While the first focused on elliptic surfaces fibered over $\mathbb{P}^1$, the latter dealt with basis of arbitrary genus. Nevertheless, both works are concerned with more general elliptic surfaces than the scope of this paper. In Kuwata's work he supposes that each component of the reducible fibers is defined over the ground field $k$. Let $E$ be the generic fiber of an elliptic surface defined over $k$ with base curve $C$. He proves that there is an explicitly computable number $m$ and an explicitly computable extension $L/k$ such that $m E(\bar{k}(C))=mE(L(C))$. Our work differs from Kuwata's in several ways. Firstly, while he focus on one unique elliptic fibration on a surface, we consider one elliptic fibration which we assume is defined over some number field $k$ and use it as a point of start to study the other elliptic fibrations present on the surface. Thus in our work, one elliptic fibration is defined over the ground field, while the others not necessarily. For that reason we are concerned with different fields of definition, namely the one of the elliptic fibration and that of the Mordell--Weil group. Secondly, we focus on an specific class of surfaces, namely K3 surfaces. The further assumption that the K3 is a double cover of an extremal rational elliptic surface guarantees that the fields of definition will be much smaller than those for arbitrary elliptic surfaces. Indeed, fields of definition of the Mordell--Weil group of an elliptic surface can be quite large, for instance in \cite{SD} Swinnerton--Dyer constructed an elliptic surface for which the field of definition of the Mordell--Weil group has degree $2^7.3^4.5$, and the degrees of the fields of definition in Kuwata's work are also much larger than the bounds obtained here. Finally, it is worth to mention that Kuwata's work deal with fields of arbitrary characteristic while we focus on number fields. We expect that our work allows generalizations to that setting and the restriction has been made for the matter of simplicity but also because some of our work builds up on Miranda and Persson's work in \cite{MiPe}, and on two of the author's paper \cite{GS}. Both settings are restricted to characteristic zero.		
	\section{Preliminaries and setting}\label{sec: setting}
	
	Let $R$ be a rational elliptic surface, i.e.~a smooth projective rational surface endowed with a relatively minimal genus one fibration. We assume throughout this article that such a fibration admits a section. We denote by $$\mathcal{E}_R:R\ra\mathbb{P}^1$$ the elliptic fibration on $R$. Let $d:C\ra\mathbb{P}^1$ be a double cover of $\mathbb{P}^1$ branched on $2n$ points $p_i$, $i=1,\ldots, 2n$. Then the fiber product $R\times_{\mathbb{P}^1} C$ is endowed with an elliptic fibration $R\times_{\mathbb{P}^1} C \ra C$, induced by $\mathcal{E}_R$. 
	
	We call the fibers $\mathcal{E}_R^{-1}(p_i)$, $i=1,\ldots, 2n$, the branch fibers.
	If all the branch fibers are smooth, then the fiber product $R\times_{\mathbb{P}^1} C$ is smooth, and we denote it by~$X$. Otherwise, $R\times_{\mathbb{P}^1} C$ is singular and we denote by $X$ its smooth model such that the elliptic fibration $\mathcal{E}_X:X\ra C$, induced by $\mathcal{E}_R$, is relatively minimal.
	
	\begin{center}
		\begin{tikzpicture}
		\matrix (m) [matrix of math nodes,row sep=3em,column sep=4em,minimum width=2em]
		{
			X & R \times_{\PP^1} C & R\\
			& C & \PP^1 \\};
		\path[-stealth]
		(m-1-1) edge [dashed] (m-1-2) 
		(m-1-1) edge node [below] {$\mathcal{E}_X$} (m-2-2)
		(m-1-2) edge (m-1-3)  
		(m-2-2) edge node [above] {$2:1$} (m-2-3) 
		(m-1-2) edge (m-2-2)
		(m-1-3) edge node [right] {$\mathcal{E}_R$} (m-2-3); 
		\end{tikzpicture}
	\end{center}

	Assume that $R$, the fibration $\mathcal{E}_R$ and the zero section $\mathcal{O}$ are all defined over a given number field $k$, which we fix once and for all.  
If the morphism $d$ is defined over $k$ then so is the fiber product, its possible desingularization $X$ and the inherited elliptic fibration $\mathcal{E}_X$. 
	
	The surface $R\times_{\mathbb{P}^1} C$ is naturally endowed with an involution, namely the cover involution of the map $R\times_{\mathbb{P}^1} C \ra R$ induced by the $2:1$ map $d:C\ra \mathbb{P}^1$. It extends to an involution $\tau\in \Aut (X)$ which is the cover involution of the generically $2:1$ cover $X\ra R$. we denote by $\pi$ the quotient map $\pi:X\to X/\tau \simeq_{\text{ bir}} R$.
		
	From now on  we make the following assumptions.
	\begin{itemize}
	\item $d:C\ra \mathbb{P}^1$ is defined over $k$,
	\item $n=1$, i.e. $d:C\ra \mathbb{P}^1$ is branched in two points. Hence $C\simeq \mathbb{P}^1$,
	\item the (two) branch fibers are reduced.
	\end{itemize}
	As a consequence of the previous assumptions we have that $X$ is a K3 surface over~$k$ (see \cite[Example 12.5]{SS}), the involution $\tau$ is non-symplectic, i.e. it does not preserve the symplectic form defined on $X$, since the quotient of a K3 by a symplectic involution is again a K3 surface (see \cite{Nikulin}), and both $\mathcal{E}_X$ and its zero section are defined over $k$. Moreover, if the branch fibers are smooth, the reducible fibers of $\mathcal{E}_X$ occur in pairs that are exchanged by $\tau$. 
	
	\begin{notation}\rm{
	 We denote by $\tau^*$ the involution induced by $\tau$ on $\mathrm{NS}(X)$.}
	\end{notation}

	We recall that, due to their geometry, i.e. trivial canonical class and regularity, K3 surfaces might admit more than one elliptic fibration, all with basis $\mathbb{P}^1$, see for instance \cite[Lemma 12.18]{SS}. Let $X$ be as above, then it admits an elliptic fibration $\mathcal{E}_X$ and at least another elliptic fibration different from $\mathcal{E}_X$ \cite[\S 8.1]{CG} and \cite[Proposition 2.9]{GS2}. One can divide the elliptic fibrations on $X$ in three different classes, depending on the action of $\tau$ on its fibers. In particular, let $\eta$ be an elliptic fibration on $X$ then, by \cite[Section  4.1]{GS}, it is
	\begin{itemize}
		\item of type 1 with respect to $\tau$, if $\tau$ preserves all the fibers of $\eta$;
		\item  of type 2 with respect to $\tau$, if $\tau$ does not preserve all the fibers of $\eta$, but maps a fiber of $\eta$ to another one. In this case $\tau$ is induced by an involution of the basis of $\eta :X\ra\mathbb{P}^1$. It fixes exactly two fibers and $\tau^*$ preserves the class of a fiber of $\eta$;
		\item of type 3, if $\tau$ maps fibers of $\eta$ to fibers of another elliptic fibration. In this case $\tau^*$ does not preserve the class of the generic fiber of $\eta$.
	\end{itemize}

The distinct elliptic fibrations on $X$ are not necessarily defined over $k$. Moreover, different fibrations might be defined over different fields. The aim of this paper is to take a first step into understanding how the action of the involution $\tau$ on the fibers of a given fibration might influence its field of definition. Throughout this paper we adopt the following definition.

\begin{definition}\label{def: fields of def}
Given $X$ as above and an elliptic fibration $\eta$ on $X$, then the smallest field extension of $k$ over which the class of a fiber of $\eta$ is defined and $\eta$ admits a section is called the field of definition of the fibration $\eta$. We denote it by $k_{\eta}$. We denote by $k_{\eta, \mathrm{MW}}$ the smallest field extension of $k_{\eta}$ over which the Mordell--Weil group of $\eta$ admits a set of generators. 
\end{definition} 

\begin{rem}
\rm{
The reader should be aware that in Def. \ref{def: fields of def} our starting data is a K3 surface $X$ constructed as a base change of a rational elliptic surface $R$. Thanks to this construction $X$ inherits an elliptic fibration from $R$ which is defined over a number field $k$. All other fields of definition that appear in this paper are (possibly trivial) field extensions of $k$. In this sense, the field of definition is unique, but when considering $X$ without this preliminary data then the field is no longer necessarily unique. Indeed, one could for instance obtain the same $X$ as a double cover of another rational elliptic surface $R'$ defined over a different field $k'$.}
\end{rem}
	\section{Rational curves on K3 surfaces}\label{sec: rational curves on K3 surfaces}
	
	Let $X$ be a K3 surface as in Section \ref{sec: setting}. In this section we study the behavior of the image by the quotient map $\pi$ of the rational curves on $X$. As in the case of elliptic curves, this behavior depends on the action of the cover involution $\tau$ on the rational curve. 
	
	\begin{lemma}\label{key lemma}
		Let $C$ be a smooth rational curve on $X$ and $D=\pi(C)$ its image on $R$. Denote by $m$ the intersection number $C\cdot \tau(C)$. Then $D$ is of one of the following types.
		\begin{itemize}
			\item[i)] A fiber component of $\mathcal{E}_R$ on $R$.
			\item[ii)] A section of $\mathcal{E}_R$.
			\item[iii)] An $m$-section of $\mathcal{E}_R$, where $m>0$.
		\end{itemize}
		Moreover, if $\pi$ is branched over two smooth fibers of $\mathcal{E}_R$ then $i)$ implies $m=0$.
	\end{lemma}
	\begin{proof}
		Let $C$ be a smooth rational curve on $X$ and $D=\pi(C)$. By the adjunction formula we have that $C^2=-2$. We consider the following cases $\tau(C)=C$ and $\tau(C)\neq C$. 
		\begin{enumerate}
			\item $\tau(C)=C$. In this case, the involution can either act as the identity on $C$ or as an involution of $C$. If the former holds then $D$ is a $(-2)$-curve on $R$ and therefore it is a component of a fiber of $\mathcal{E}_R$. If $\tau$ acts as an involution on $C$ then since $\pi_*(C)=2D$, we have that $2D^2=C^2=-2$. Hence $D^2=-1$, and in particular $D$ is a section of $\mathcal{E}_R$. 
			\item $\tau(C)=C'\neq C$. Then $m\geq 0$, then $\pi^*(D)^2=2D^2=(C+C')^2=-4+2m$. Hence $D^2=m-2$. By the adjunction formula we have that $D(-K_R)=m$. To conclude it is enough to recall that the class of a fiber of the elliptic fibration on $R$ is given by $-K_R$. Thus, $D$ is an $m$-section of $\mathcal{E}_R$ if $m>0$, or a fiber component of $\mathcal{E}_R$ if $m=0$.
		\end{enumerate}
		Moreover, if $\pi$ is branched over two different smooth fibers,  $\tau(C)=C$ implies that $\tau$ is an involution of $C$, and thus $D$ is a section of the elliptic fibration $\mathcal{E}_R$. Hence if $D$ is a component of a fiber one must have $\tau(C)\neq C$, i.e., case (2) with $m=0$. 
	\end{proof}
	
The next lemma deals with rational curves on $X$ that come from sections defined over $k$ of the elliptic fibration $\mathcal{E}_R$. As sections do not split on the double cover we show that their inverse image is a irreducible curve defined over $k$. 

	\begin{lemma}\label{lemma: sections on R and induced curves on X} Let $P_R$ be a section of $\mathcal{E}_R:R\ra\mathbb{P}^1$ that is defined over $k$, then $P_X:=\pi^{-1}(P_R)$ is an irreducible smooth rational curve of $X$ and $\tau(P_X)=P_X$. In particular $P_X$ is defined over $k$.\end{lemma}
	\proof If $P_R$ is a section of an elliptic fibration on a rational surface then it meets the branch locus of $R \times_{\mathbb{P}^1} \mathbb{P}^1\ra R$, which is given by two fibers, in two points. Thus its inverse image is a $2:1$ cover of a rational curve branched in two points, i.e. either an irreducible smooth rational curve or the union of two smooth rational curves meeting in two points. If the inverse image of $P_R$ is the union of two curves, say $P_1$ and $P_2$, we have $\pi^*(P_R)=P_1+P_2$. Since the inverse image of a fiber $F_R$, which is not a branch fiber, consists of two disjoint fibers, we have $\pi^*(F_R)=(F_1+F_2)$. But then we would have $\pi^*(F_R)\pi^*(P_R)=2=(F_1+F_2)(P_1+P_2)=2(F_1P_1)+2(F_1P_2)$, where we used that $F_1$ and $F_2$ are linearly equivalent, since they are fibers of the same fibration on $X$. This would implies that either $P_1$ or $P_2$ is a component of a fiber, which is not possible, because they intersect in two points which lie in two different fibers, namely the ramification fibers. We conclude that $\pi^{-1}(P_R)$ is a smooth rational curve.	
	Even if one has to blow up some points to obtain $X$ from $R\times_{\mathbb{P}^1} \mathbb{P}^1$, the strict transform of the inverse image of $P_R$, which we denote by $P_X$, remains irreducible and thus $\tau(P_X)=P_X$. Since the double cover map $d$ is assumed to be defined over $k$ and so are the points that one has to possibly blow up, we have that $P_X$ is also defined over $k$.
	\endproof
	
\section{Extremal rational elliptic surfaces}\label{sec: extremal RES}

In what follows we analyze the arithmetic of extremal rational elliptic surfaces defined over $k$. Let us recall that an extremal rational elliptic surface has Mordell--Weil rank equal to $0$, and thus only finitely many sections, i.e. $(-1)$-curves.
		
\begin{lemma}\label{lemma: quadratic extension}
Let $R$ be an extremal rational elliptic surface defined over $k$. Assume that all reducible fibers of the elliptic fibration are distinct. Then the N\'{e}ron--Severi group $\mathrm{NS}(R)$ admits generators defined over a field extension of $k$ of degree at most 2.
\end{lemma}
\begin{proof}
There are two main ingredients in the proof of the statement. The first one is the Shioda--Tate formula which tells us that 
\[
\mathrm{NS}(R)/ T \simeq \mathrm{MW}(\mathcal{E}_R),
\]
where $T= \langle \mathcal{O}, F \rangle \oplus \sum_{\substack{v \in \text{reducible fibers}\\ i \in S_v}} \Theta_{v,i}$, with $\Theta_{v,i}$ denoting  the components of the reducible fiber $\mathcal{E}_R^{-1}(v)$, $S_v=\{1,\cdots, n_v-1\}$ and, since the surface is extremal, $\mathrm{MW}(\mathcal{E}_R)$ is a finite group. The second is the fact that the absolute Galois group $G_{\bar{k}}$ acts on $\mathrm{NS}(R)$ preserving the intersection pairing. 

 Recall that both the zero section $\mathcal{O}$ and the class of a smooth fiber $F$ are defined over $k$. A reducible fiber with exactly two components has each component defined over $k$ since the component that intersects the zero section is preserved. Thus in what follows we can focus on reducible fibers with at least three components. By the hypothesis on the reducible fibers being distinct, there are at most two such fibers, say $F_{v_1}$ and $F_{v_2}$, see the table in \cite[Thm. 4.1]{MiPe}. Assume w.l.o.g that $F_{v_1}$ is the fiber with more reducible components. Each reducible fiber is globally defined over $k$ because, by assumption, it is unique. Hence its trivial component is also defined over $k$. Since the latter intersects at most two other components, these are $G_{\bar{k}}$-conjugate and as a pair they form a $G_{\bar{k}}$-orbit. The same happens to all other components that are not defined over $k$. Let $k_R/k$ be the quadratic extension over which the fiber components of $F_{v_1}$ are defined. We show that each section is defined over $k_R$. The Mordell--Weil group is globally defined over $k$ since its elements are precisely the $(-1)$-curves in the N\'eron--Severi group. Moreover because each section $C$ intersects transversally a unique fiber component of $F_{v_1}$, the point of intersection is mapped by any element in $G_{\bar{k}}$ to another point of intersection of a component of $F_{v_1}$ and a section. Since a component of a fiber is mapped by $G_{\bar{k}}$ either to itself or to a unique other fiber component defined over $k_R$, the intersection point is also defined over $k_R$. Thus $C$ is a rational curve with a $k_{R}$-point and hence it is  also defined over $k_{R}$. It remains to show that the components of $F_{v_2}$ are defined over $k_{R}$. This follows from the fact that after contracting the sections and certain fiber components of $F_{v_1}$ we reach either $\mathbb{P}^2$ or $\mathbb{P}^1\times \mathbb{P}^1$. The components of $F_{v_2}$ are thus rational curves with $k_{R}$-points that correspond to the contracted curves, and hence are defined over $k_{R}$ as well.
\end{proof}

\begin{example}
The extremal rational elliptic surface with Weierstrass equation
\[
y^2=x^3  - 3(t^2- 3)(t-2)^2 x+ t(2t^2-9)(t-2)^3
\]
has reducible fibers of types $I_1^*$ and $I_4$. Its Mordell--Weil group is $\mathbb{Z}/4\mathbb{Z}$ with two sections defined over $\mathbb{Q}$, namely $[0, 1, 0]$ and $ [t^2-2t, 0,1]$,
 and two conjugate sections, namely $[(t - 3)(t-2), \pm 3 \sqrt{3}(t-2)^2,1]$, which are defined over a quadratic extension. The reader can find this example as $X_{141}$ in \cite[Table 5.2]{MiPe}.
 \end{example}
The next example shows that the hypothesis on the distinct reducible fibers is indispensable in Lemma \ref{lemma: quadratic extension}.
 \begin{example}
The extremal rational elliptic surface with Weierstrass equation
 \[
 y^2= x^3 + (3t^4+24t )x+ 2t^6+40t^3-16
 \]
 has four reducible fibers of type $I_3$. Its Mordell--Weil group is defined over a biquadratic extension $\mathbb{Q}(i, \sqrt{3})$. This corresponds to the surface $X_{3333}$ in \cite[Table 5.3]{MiPe}. See also Remark \ref{RmkNS}, iii).
 
 \end{example}

\begin{notation}\label{not: compositum}\rm {In what follows, we keep the notation introduced in Lemma \ref{lemma: quadratic extension} and denote by $k_R$ the extension of $k$ over which the N\'eron--Severi group $\mathrm{NS}(R)$ admits a set of generators given by fiber components and sections of the elliptic fibration on $R$, and by $G_R$ the Galois group $\Gal(k_R/k)$. We keep the subscript $R$ for the Galois group to reinforce the dependence on the surface. By Lemma \ref{lemma: quadratic extension}, if the Kodaira types of the reducible fibers of $\mathcal{E}_R$ are different  then $k_R/k$ has degree at most 2.}
\end{notation}

\begin{rem}\label{RmkNS}\rm{
\indent\par
\begin{itemize}
\item[i)] Certain configurations of reducible fibers force the Galois group $G_R$ to be trivial. Thus such surfaces always admit a set of generators for their N\'eron--Severi group over the ground field $k$. This holds for instance for any rational elliptic surface over $k$ which has reducible fiber configurations $(II^*)$, $(III^*,I_2)$, $(III^*,III)$ or $(I_4^*)$; see the proof of Lemma~\ref{models over QQ} or \cite[Cor. 4.4]{Sal16}.
\item[ii)] Five out of sixteen configurations of reducible fibers on extremal rational elliptic surfaces, namely $(2I^*_0), (2I_5), (2I_4,2I_2), (I_2^*, 2I_2)$ and $(4I_3)$ do not satisfy the hypothesis of Lemma \ref{lemma: quadratic extension}, see \cite[Theorem 4.1]{MiPe}. 
\item[iii)] Extremal rational elliptic surfaces with repeated reducible fibers have their N\'eron--Severi group defined, in general, over extensions of larger degree. For instance, a rational elliptic surface with reducible fiber configuration $(2I_5)$ has, in general, its N\'eron--Severi group defined over an extension of degree four, with cyclic Galois group (see the proof of Lemma \ref{models over QQ}), while a surface $R_0$ with $(2I_0^*)$ has, in general, $\mathrm{NS}(R_0)$ defined over an extension of the ground field with Galois group given by the dihedral group of order 12. Indeed, the Galois group is generated by an involution which preserves each section and switches the two $I_0^*$-fibers and by  $\mathfrak{S}_3$ which preserves the fibers, and permutes the non-trivial elements of $\mathrm{MW}(R_0)=(\Z/2\Z)^2$.
\end{itemize}}
\end{rem}

\vspace{5pt}

\subsection{Minimal models for extremal RES over \texorpdfstring{$k$}{k}}\hspace{0pt}

\vspace{5pt}

We recall that every rational elliptic surface defined over and algebraically closed field of characteristic zero can be obtained as the blow-up of the base points of a pencil of generically smooth cubics, \cite[\S 5.6.1]{CoDol} or \cite[Lemma IV.1.2.]{Mi}. This fact clearly does not hold, in general, over a number field $k$. For instance, the blow-up of the base point of the anti-canonical linear system of a $k$-minimal del Pezzo surface of degree one is a rational elliptic surface defined over $k$ which does not admit a blow down to $\mathbb{P}^2$ as it is clearly not even $k$-rational. On the other hand, if one restricts our attention to extremal rational elliptic surfaces then one can show that they are always $k$-rational, with possible exception given by those with reducible fiber configuration $(2I_0^*)$\footnote{These can be $k$-birational to a $k$-minimal Ch\^atelet surface depending on whether the elliptic fibration has a 2-torsion section over $k$ or not.}. Still this is not enough to assure that they can be obtained as a blow-up of the projective plane. Indeed, we provide an example in Proposition \ref{prop: Galois of R9} for which this does not hold. Nonetheless, we obtain a quite intriguing fact, namely that with a possible exception of surfaces with configuration $(2I_0^*)$, all extremal rational elliptic surfaces can be obtained, over the ground field, as a blow-up of base points of a pencil of genus one curves in $\mathbb{P}^2$ or $\mathbb{P}^1\times\mathbb{P}^1$, in Lemma \ref{models over QQ}. Despite its simple proof, this intriguing fact is not in the literature and likely not known to many experts. 

Since an extremal rational elliptic surface has finite Mordell--Weil group, it has only finitely many curves of negative self-intersection \cite[Proposition VIII.1.2]{Mi}. The Galois group $G_R$ acts on $\mathrm{NS}(R)$ preserving the intersection pairing. Since, by hypothesis, the zero section of the fibration $\mathcal{E}_R$ is defined over $k$ it is always preserved by $G_R$.

From now on we will use the following notation for the irreducible components of a reducible fiber: the component which intersects the zero section will be denoted by $C_0$; in a fiber of type $I_n$, the components $C_i$, $i\in\Z/n\Z$ are numbered requiring that $C_iC_{j}=1$ if and only if $|i-j|=1$.
\begin{lemma}\label{models over QQ}
Let $R$ be an extremal rational elliptic surface defined over $k$ with at most one non-reduced fiber. Then $R$ is $k$-isomorphic to the blow-up of the base points of a pencil of cubic curves in $\mathbb{P}^2$ or a pencil of curves of bidegree $(2,2)$ in $\mathbb{P}^1\times \mathbb{P}^1$. In particular, such surfaces are always $k$-rational.
\end{lemma}
\begin{proof}
We recall that the Galois group $G_R$ preserves the zero section, maps a fiber of a certain Kodaira type to a fiber of the same Kodaira type and maps sections to sections. The consequences are the following:
	\begin{enumerate}
		\item[i)] If $\mathrm{MW}(\mathcal{E}_R)=\{0\}$ or $\mathrm{MW}(\mathcal{E}_R)=\Z/2\Z$, then $G_R$ maps each section to itself;
		\item[ii)] If every reducible fiber is of different Kodaira type, then $G_R$ maps the zero component of each fiber to itself;
		\item[iii)] If both i) and ii) are satisfied, then $G_R$ is trivial since in that case the fiber with most components is a non-reduced fiber of type $II^* ,III^*$ or $I_4^*$ (see the table in \cite[Thm. 4.1]{MiPe}) and each component is  preserved by the Galois group because the zero section and the two torsion are preserved and defined over $k$; 
		\item[iv)] If there is a fiber which is preserved by $G_R$ as, for example, in case ii) and it is either of type $I_n$ or of type $IV^*$, then $G_R$ restricted to that fiber and to $\mathrm{MW}(\mathcal{E}_R)$ acts trivially or as the hyperelliptic involution because it has to preserve the intersection properties of the components of the reducible fiber of type $I_n$.
	\end{enumerate}

Using these properties of $G_R$, one is able to find an explicit contraction $\gamma$ defined over $k$, which maps a rational elliptic surface $R$ either to $\mathbb{P}^2$ or to $\mathbb{P}^1\times \mathbb{P}^1$ for all the extremal rational elliptic surfaces $R$ with reducible fiber configuration different from $(2I_0^*)$.

{\bf Fibrations $(II^*,II)$, $(II^*,2I_1)$, $(III^*,III)$, $(III^*,I_2,I_1)$, $(I_4^*,2I_1)$}: $G_R$ is trivial because iii) in the previous list is satisfied. One first contracts all the sections, then contracts the image of the components of the fibers $II^*$, $III^*$, $I_4^*$, respectively, that are the $(-1)$-curves after the previous contractions. One iterates this process in order to contract 9 curves. The composition of all these contractions is a map $R\ra\mathbb{P}^2$, defined over $k$.

{\bf Fibrations $(IV^*,IV)$, $(IV^*,I_3, I_1)$}: by iv), $G_R$ acts trivially or coincides with the hyperelliptic involution. After contracting all the sections, one obtains three $(-1)$-curves in the image of the $IV^*$-fiber. One is preserved by $G_R$, the other two might be exchanged by it. After contracting these three curves, one is in a similar situation, i.e. there are three $(-1)$-curves, forming two or three orbits for $G_R$. After contracting also these three curves, one obtains a $k$-rational map from $R$ to $\mathbb{P}^2$.

{\bf Fibrations $(I_9,3I_1)$, $(I_8,I_2, 2I_1)$, $(I_6,I_3,I_2)$}: by $(iv)$, $G_R$ acts trivially or coincides with the hyperelliptic involution. First one contracts all the sections. Then one contracts some curves in the image of the fibers of type $I_9$, $I_8$ and $I_6$ respectively, but not in the other reducible fibers. For the fiber $I_9$ one contracts the images of the components $C_0$, preserved by $G_R$, and of $C_3$ and $C_6$, which are either fixed or switched by $G_R$; after that one contracts the images of the curves $C_2$ and $C_7$, which are also either fixed or switched by $G_R$.
For the fiber $I_8$ one contracts the images of components $C_0$ and $C_4$, which are preserved by $G_R$, and of the components $C_2$ and $C_6$, which are either fixed or conjugate under $G_R$.
For the fiber of type $I_6$ one contracts the images of components $C_0$ and $C_3$, which are preserved by $G_R$. In all the cases one obtains a $k$-rational map from $R$ to $\mathbb{P}^1\times\mathbb{P}^1$.

{\bf Fibrations of type $(4I_3)$ and $(2I_4,2I_2)$}: in both these cases there are many sections, namely 9 sections in case $(4I_3)$ and 8 in case $(2I_4,2I_2)$. Since the torsion sections are disjoint and $G_R$ preserves $\mathrm{MW}(\mathcal{E}_R)$, one can contract simultaneously all the sections. This produces a $k$-rational map to $\mathbb{P}^2$ in case $(4I_3)$ and to $\mathbb{P}^1\times\mathbb{P}^1$ in case $(2I_4,2I_2)$.

{\bf Fibration of type $(2I_5, 2I_1)$}: we have $G_R\subseteq \Z/4\Z$, and if $G_R=\Z/4\Z$ then the action of the generator of $G_R$ is the following. $t_0\ra t_0$, $C_{0}^{(1)}\leftrightarrow C_0^{(2)}$, $t_1\ra t_3\ra t_4\ra t_2\ra t_1$, $C_{1}^{(1)}\ra C_{1}^{(2)}\ra C_4^{(1)}\ra C_4^{(2)}$, $C_{2}^{(1)}\ra C_{2}^{(2)}\ra C_3^{(1)}\ra C_3^{(2)}$, where $C_i^{(j)}$ the $i$-th component of the $j$-th fiber of type $I_5$.
To obtain a $k$-rational map to $\mathbb{P}^2$, one first contracts all the sections, and then one contracts the components $C_{1}^{(1)}$, $C_{1}^{(2)}$, $C_4^{(1)}$, $C_4^{(2)}$, which form an orbit if $G_R=\Z/4\Z$. 

{\bf Fibration of type $(I_2^*,2I_2)$ and $(I_1^*,I_4, I_1)$}: one contracts first the four sections and then the images of the four simple components of the fiber of type $I_i^*$. This gives a $k$-rational map to $\mathbb{P}^1\times\mathbb{P}^1$.\end{proof}

\begin{proposition}\label{differentblowdowns}
 Let $R$ be a semi-stable extremal rational elliptic surface defined over $k$ and $m$ the order of the Mordell--Weil group. Then the following holds.
 \begin{itemize}
  \item[i)] If $m$ is odd and $R$ has a unique reducible fiber then $R$ admits a contraction over $k$ to $\mathbb{P}^1\times \mathbb{P}^1$.
  \item[ii)] If $m$ is odd and $R$ has at least two reducible fibers then $R$ admits a contraction over $k$ to $\mathbb{P}^2$.
  \item[iii)] If $m$ is even then $R$ admits a contraction over $k$ to $\mathbb{P}^1\times\mathbb{P}^1$.
 \end{itemize}
\end{proposition}
\begin{proof} The result follows by the proof of the previous lemma. Indeed, if $R$ is a semi-stable extremal elliptic fibration and $m$ is odd, then the fibration on $R$ is one of the following: $(I_9,3I_1)$, $(2I_5,2I_1)$, $(4I_3)$. The first fibration corresponds to case $i)$ and can be contracted to $\mathbb{P}^1\times\mathbb{P}^1$, for every action of $G_R$. The other two fibrations correspond to the case $ii)$ and it was already proved that they can be contracted to $\mathbb{P}^2$.
		
If $m$ is even (case $iii)$), then the fibration on $R$ is one of the following: $(I_8,I_2,2I_1)$, $(I_6, I_3, I_2)$, $(2I_4, 2I_2)$ and in the proof of the previous lemma is shown that all of them can be contracted to $\mathbb{P}^2$. \end{proof}

\begin{rem}\rm{
The converse of the different cases in Proposition \ref{differentblowdowns} is not always true; some of the surfaces treated in Lemma \ref{models over QQ} can be contracted, over an algebraically closed field, to both $\mathbb{P}^2$ and $\mathbb{P}^1\times\mathbb{P}^1$. Whether or not these surfaces can be contracted to both $\mathbb{P}^2$ and $\mathbb{P}^1\times\mathbb{P}^1$ over $k$ as well depends on the action of $G_R$, and in particular on the action of the hyperelliptic involution on the reducible fibers. See Proposition~\ref{prop: Galois of R9} and Figure \ref{blowdown I9}, where we show this for a surface with fibers $(I_9,3I_1)$.}\end{rem}

\section{Double covers of extremal rational elliptic surfaces}\label{sec: double covers on ext. RES}

In the rest of this article we consider K3 surfaces that are double covers of extremal rational elliptic surfaces defined over $k$ and branched on two smooth $G_{\bar{k}}$-conjugate fibers. Let $X$ be such a surface. Recall that since the extremal rational elliptic surfaces considered here\footnote{We exclude rational elliptic surfaces with $(2I_0^*)$. These have a 1-dimensional moduli space.} are rigid, their K3 double covers have a 2-dimensional moduli space, as each branch point is allowed to vary in $\mathbb{P}^1$. 

In this section we show that the field over which a genus one fibration on $X$ admits a section depends on the action of the cover involution on the fibers of the genus one fibration. 

\begin{notation}
Let $R$ and $X$ be as above and $t_1,\cdots, t_m \in \mathbb{P}^1_k$ points over which the reducible fibers of $R$ are located. Since the base change map $X \rightarrow R$ is branched only over smooth fibers, there are two distinct points above each $t_i$. Then $\tau$ restricted to the pair of fibers of $\mathcal{E}_X$ above each $t_i$ is a field homomorphism, which we denote by $\sigma_i$. We denote by $k_{\tau}$ the Galois field extension of $k$ whose Galois group is generated by $\sigma_1, \cdots, \sigma_m$. By construction $k_{\tau}/k$ is an extension of even degree dividing $2^m$.
We denote by $k_{R,\tau}$ the compositum of the fields $k_R$ and $k_{\tau}$.
\end{notation}

\begin{lemma}\label{generators_NS(X)} Let $R$ be an extremal rational elliptic surface as above and $X$ a generic member of the 2-dimensional family given by double covers of $R$ branched in two smooth fibers. Then $\mathrm{NS}(X)$ admits a set of generators over $k_{R,\tau}$.
\end{lemma}
\begin{proof} Since the N\'eron--Severi group has rank 10 and the Mordell--Weil group has rank zero, it follows from the Shioda--Tate formula that the reducible fibers of an extremal rational elliptic surface $R$ have in total 8 components contributing to the set of generators of $\mathrm{NS}(R)$. Since $X$ is a double cover of an extremal rational elliptic surface $R$ branched on smooth fibers, the reducible fibers of the inherited fibration $\mathcal{E}_X$ contribute with 16 components to a set of generators of $\mathrm{NS}(X)$. If $X$ is generic among such surfaces then it lies in a 2-dimensional family and hence $\mathrm{NS}(X)$ has rank 18 and is generated by fiber components, the zero section and a smooth fiber of $\mathcal{E}_X$. All such curves are defined at most over $k_{R, \tau}$.
\end{proof}

\begin{theorem}\label{theo: types and field of definition}
Let $R$ be an extremal rational elliptic surface defined over $k$ such that its reducible fibers are all of distinct Kodaira types. Let $X$ be a K3 surface obtained as a double cover of $R$ branched on two smooth fibers conjugate under $G_{\bar{k}}$, $\tau$ the cover involution and $\eta$ a genus 1 fibration on $X$. Then the following hold.
\begin{itemize}
\item[i)] If $\eta$ is of \emph{type 1} w.r.t. $\tau$ then $\eta$ is defined over $k_R$ and admits a section over $k_{R, \tau}$.
\item[ii)] If $\eta$ is of \emph{type 2} w.r.t. $\tau$ then it is defined and admits a section over $k$.
\end{itemize}
\end{theorem}

\begin{proof}
For ii) notice that because the branch locus is smooth there is only one fibration of type 2, namely the one induced by the elliptic fibration on $R$. Indeed, different fibrations of type 2 correspond to different contractions of $(-1)$-curves in $X/{\tau}$ that are components of non-relatively minimal elliptic fibrations. Since the branch locus is smooth there are no $(-1)$-curves to be contracted and, in particular, $X/{\tau} \simeq R$. Since the double cover morphism is defined over $k$ so is the induced elliptic fibration on $X$ and the zero section inherited from $R$.
If $\eta$ is of type 1 then each fiber is the pull-back of a conic\footnote{A conic is a rational curve $C$ such that $C\cdot (-K_R)=2$.} in $R$ \cite[Theorem 4.2]{GS}. Let $C$ be such a conic. Since $\mathrm{NS}(R)$ is generated by curves defined over $k_R$ then the class of $C$ has a divisor $C_0$ whose components are defined over $k_R$. Moreover, as the fibers of $\eta$ are fixed by $\tau$, the pull-back $C_0$ is also defined over $k_R$. Its class moves in $X$ giving the elliptic fibration $\eta$. 
\end{proof}
The fibrations of type 3 are certainly more difficult to study by using the geometry related with $R$. Indeed, even if $X$ is a double cover of $R$, the fibrations of type 3 are not easily related with the geometry of $R$, by definition, since they are not preserved by the cover involution. But, one is still able to prove that certain fibration of type 3 are defined on certain fields, if one is able to find components of their reducible fiber is a proper way, as observed in the next Remark.

\begin{rem}\label{lemma: type 3}{\rm Since the irreducible components of reducible fibers and of the sections  of the elliptic fibration on K3 surface are rational curves, they are rigid in their class. So if their class is defined over a certain field, say $k_{R,\tau}$, and they are irreducible curves, then they are defined over $k_{R,\tau}$. Suppose now that the N\'eron--Severi group is defined over $k_{R,\tau}$ and it is generated by a certain set of classes of irreducible rational curves. If the union of some of these curves is a reducible fiber $F$ of a fibration $\eta$, then the reducible fiber $F$ and its class are defined over $k_{R,\tau}$. In particular the fibration $\eta$ is defined on $k_{R,\tau}$ and if also a section of $\eta$ can be found among the generators of the N\'eron--Severi, then $\eta$ is an elliptic fibration on $k_{R,\tau}$.
		
		So, in order to prove that a fibration of type 3 defined on a K3 surface satisfying the assumptions of Theorem \ref{theo: types and field of definition}, is defined over $k_{R,\tau}$, it suffices to find among the generators of $\mathrm{NS}(X)$ a configuration of $(-2)$-curves which corresponds to a reducible fiber of $\eta$. }
\end{rem}

\begin{rem}{\rm
We believe that it is always possible to find a fibration of type 3 as in the previous remark, at least for the K3 surfaces $X$ as in Thereom \ref{theo: types and field of definition}. We are able to prove this for all the elliptic fibrations of type 3 on the surfaces considered in Sections \ref{sec: classification of the fibrations of X9}, and \ref{sec: classification X2 X3 X4} of this paper. Hence for all the surfaces considered in this paper, we have that the field of definition of the elliptic fibrations on the K3 surfaces $X$ as in Theorem \ref{theo: types and field of definition} are at most biquadratic extension of $k$, by the explicit description of the elliptic fibration and the Remark \ref{lemma: type 3}.
}\end{rem}
\begin{rem}\rm{
Certain sections on elliptic K3 surfaces as above might be defined over a smaller subfield of $k_{R,\tau}$ that contains $k$. See, for instance, the fifth column of lines 2, 3, 4, 9, 11 and 12 in Table \ref{eq: tableof types of the elliptic fibrations on X9}.
}
\end{rem}

Following the geometric classification of extremal rational elliptic surfaces by Miranda and Persson \cite[Theorem 4.1]{MiPe}, we notice that, among those surfaces, only four of them have only one reducible fiber, namely $(I_9,3I_1)$, $(II^*,II)$, $(II^*,2I_1)$ and $(I_4^*,2I_1)$. 
From a lattice theoretic point of view the surfaces with singular fibers $(II^*,II)$ and $(II^*,2I_1)$ are the same since, from that perspective, only the reducible fibers matter. Moreover, they share the same properties of interest to us, namely reducible fibers and fields of definition of components of fibers and thus we denote both of them by $R_2$. In the following sections, we study those extremal rational elliptic surfaces, denoted by $R_9$, $R_2$, and $R_4$ and their corresponding K3 surfaces $X_9$, $X_2$, and $X_4$, respectively. We also study the surface $R_3$ which has two reducible fibers $(III^*,III)$ and its generic K3 cover $X_3$. The justification for considering $R_3$ as well is the fact that the surface $X_4$ occurs also as double cover of $R_3$ and hence $X_3$ and $X_4$ belong to the same family of K3 surfaces.

\subsection{Arithmetic models of extremal rational elliptic surfaces}\label{uniquenessofmodels}
Over algebraically close fields, all rational elliptic surfaces can be obtained by the blow up of the base points of a pencil of genus 1 curves in the projective plane. Over a number field $k$, this not longer holds true. Nevertheless, if one restricts attention to extremal rational elliptic surfaces, we have shown in Lemma \ref{models over QQ} that, with one possible exception, they can be obtained as a blow up of a pencil of genus 1 curves in the plane or in the ruled surface $\mathbb{P}^1\times \mathbb{P}^1$. The realization of the blow down of an extremal rational elliptic surface $R$ to either rational minimal model is connected to, but not always determined by, the Galois group $G_R$ introduced in Notation \ref{not: compositum}. More precisely, given singular fiber configurations on an extremal rational elliptic surface might entail more than one possible action of the Galois group $G_{\bar{k}}$ on its fiber components and hence, with a few exceptions, it does not make sense anymore to speak about \textbf{the} extremal rational elliptic surface with a given configuration as one does over algebraically closed fields. In what follows we keep the notation $R_i$ and $X_i$ for \textbf{a} surface with fiber configuration described in the previous paragraph. We study what are the possible actions of $G_{\bar{k}}$ on each configuration. We show, in Propositions \ref{prop: Galois of R9} and \ref{prop: Galois R2, R3 and R4} respectively, that $R_9$ might admit two possible actions, while  $R_2, R_3, R_4$ always admit a unique action. 

	\section{The surfaces \texorpdfstring{$R_9$}{R9} and \texorpdfstring{$X_9$}{X9}}\label{sec: surfaces R9 and X9}
	
Let $R_9$ be an extremal rational elliptic surface with one reducible fiber of type $I_9$ and $X_9$ a K3 surface obtained by a double cover of $R_9$ branched in two smooth $G_{\bar{k}}$-conjugate fibers. 	
In this section, we classify all the possible fibrations of the K3 surface $X_9$ and determine their types with respect to the cover involution $\tau_9$, a field over which the class of a fiber is defined and a field over which the Mordell--Weil group is defined.

\subsection{Negative curves on \texorpdfstring{$R_9$}{R9}}\hspace{0pt}

\vspace{5pt}

Recall that the configuration $I_9$ is given by 9 smooth rational curves meeting in a cycle with dual graph $\tilde{A_8}$ (see \cite[Table I.4.1]{Mi}\footnote{Though this table contains a typo, namely a fiber of $I_n$ has dual graph $\tilde{A}_{n-1}$}). The singular fibers of $R_9$ are $I_9+3I_1$ and the Mordell--Weil group is $\Z/3\Z=\{\mathcal{O}, t_1, t_2\}$, where $\mathcal{O}$ is the zero section and $t_1$ and $t_2$ are 3-torsion sections. The N\'eron--Severi group of $R_9$ contains also the classes of the irreducible components of the unique reducible fiber, denoted by $C_0, C_1,\ldots, C_8$. The intersections which are not trivial are the following
\begin{align*} 
&C_i^2=-2; &C_iC_{j}=1\mbox{ iff }|i-j|=1;\\
&\mathcal{O}C_0=t_1C_3=t_2C_6=1; &\mathcal{O}^2=t_1^2=t_2^2=-1.
\end{align*}

The following result tells us that $R_9$ can always be obtained as the blow-up of the eight base points on a pencil of curves of bi-degree (2,2) in $\mathbb{P}^1\times \mathbb{P}^1$, and that if the Galois group $G_{R_9}$ fixes each 3-torsion section then $R_9$ can also be obtained as the blow-up of the nine base points of a pencil of cubics in $\mathbb{P}^2$ (see also Lemma \ref{models over QQ}). Both blow-ups occur in multiple points, i.e., points with assigned multiplicities.

\begin{proposition}\label{prop: Galois of R9}
	If for every $g\in G_{R_9}=\Gal (k_{R_9}/k)$ we have $g(t_1)=t_1$, then $G_{R_9}=\{ id\}$ and $R_9$ can be contracted both to $\mathbb{P}^2$ and to $\mathbb{P}^1\times \mathbb{P}^1$. 	
	If there exists at least one $g\in G_{R_9}$ such that $g(t_1)\neq t_1$, then $g(t_1)= t_2$, $G_{R_9}=\Z/2\Z=\langle g\rangle$ and $g$ is the elliptic involution $\iota_{R_9}$ restricted to the fiber $I_9$. In this case $R_9$ can be contracted to $\mathbb{P}^1\times \mathbb{P}^1$ but not to $\mathbb{P}^2$.
\end{proposition}
\proof 
Let $F$ be the class of a fiber of $\mathcal{E}_{R_9}$. Since $F$ is preserved by $G_{R_9}$, for each $g\in G_{R_9}$ we have $1=t_1F=g(t_1)g(F)$ and thus $g(t_1)$ is necessarily a section. It is different from $\mathcal{O}$ as the latter is fixed by $G_{R_9}$. Hence either $g(t_1)=t_1$ or $g(t_1)=t_2$. 
We begin with $g(t_1)=t_1$. In that case $g(t_2)=t_2$ and since $t_1$ intersects the fiber component $C_3$ and $t_2$ intersects $C_6$, we have $g(C_3)=C_3$ and $g(C_6)=C_6$. Since each other fiber component intersects one among $C_0,C_3$ and $C_6$, it is also fixed by $g$. Hence $G_{R_9}$ is trivial.
We  pass to the case $g(t_1)=t_2$. This implies that $g(C_3)=C_6$. The fiber components intersecting $C_3$ and $C_6$ must be switched by $g$ and, a posteriori, so must $C_1$ and $C_8$. We have $g(C_i)=C_{9-i}$. Hence, in that case, $G_{R_9}$ has order 2 and is generated by the elliptic involution.

Let us now consider the contraction of the $(-1)$-curves on $R_9$, i.e., the sections $\mathcal{O}$, $t_1$ and $t_2$. The reader might find helpful to follow Figure \ref{blowdown I9} in parallel. First one contracts the three sections, which are all disjoint and form either 3 or 2 orbits for the action of $G_{R_9}$, depending on whether $G_{R_9}$ is $\{ id\}$ or $\Z/2\Z$. Let us denote by $\beta_1:R_9\ra R'$ this contraction. The curves $\beta_1(C_0)$, $\beta_1(C_3)$, $\beta_1(C_6)$ are disjoint $(-1)$-curves of $R'$ and form 2 or 3 orbits with respect to $G_{R_9}$.  Secondly, we call $\beta_2:R'\ra R''$ the contraction of these three curves. The curves $\beta_2(\beta_1(C_i))$ for $i=1,2,4,5,7,8$ are $(-1)$-curves on $R''$. The curves $\beta_2(\beta_1(C_2))$ and $\beta_2(\beta_1(C_7))$ form 1 or 2 orbits with respect to $G_{R_9}$. 
Hence they can be contracted in order to obtain a minimal surface. Let us denote by $\beta_3:R''\ra R'''$ this contraction. Then  $R'''$ is $\mathbb{P}^1\times\mathbb{P}^1$, the curves $\beta_3(\beta_2(\beta_1(C_1)))$ and $\beta_3(\beta_2(\beta_1(C_5)))$ are curves of bidegree $(1,0)$ in $\mathbb{P}^1\times\mathbb{P}^1$ and the curves $\beta_3(\beta_2(\beta_1(C_4)))$ and $\beta_3(\beta_2(\beta_1(C_8)))$ are curves of bidegree $(0,1)$. Hence the image of the reducible fiber $I_9$ is a reducible curve of bidegree $(2,2)$ in $\mathbb{P}^1\times\mathbb{P}^1$. 
There is another possible choice of curves to contract on $R''$ in order to obtain a minimal surface. If $G_{R_9}=\{ id \}$, one can contract the curves $\beta_2(\beta_1(C_1))$, $\beta_2(\beta_1(C_4))$, $\beta_2(\beta_1(C_7))$ obtaining $\mathbb{P}^2$ as minimal surface. But these curves do not form an orbit for $G_{R_9}$ if $G_{R_9}=\Z/2\Z$, hence this contraction is allowed only if $G_{R_9}$ is trivial. 
\endproof

Figure \ref{blowdown I9} shows the contractions $\beta_1,\;\beta_2$, and $\beta_3$ of the fiber $I_9$ as in the proof of Proposition \ref{prop: Galois of R9}. Black lines represent curves defined over $k$. Lines of the same color (not black) represent curves that are conjugate under the action of $G_{R_9}$ if $G_{R_9}\neq\{id\}$; of course if $G_{R_9}=\{id\}$ then all curves are defined over $k$. Dotted lines represent $(-1)$-curves, lines with label $0$ represent curves with self-intersection 0, and all other lines represent $(-2)$-curves.  
\begin{center}
\begin{figure}[!h]
\begin{tikzpicture}[scale = 1]
  \foreach \x /\alph /\name in {270/a/$ $, 310/b/$ $, 350/c/$ $,  30/d/$ $,  70/e/$ $,  110/f/$ $, 150/g/$ $, 190/h/$ $, 230/i/$ $}{
  \node[circle,inner sep=0,minimum size=0.0pt] (\alph) at (\x:1.5cm) {\name}; }
\node [label={[shift={(0.2,-0.6)}]\small{$C_8$}}] at ($(f)!0.5!(g)$) {};
\node [label={[shift={(0,-0.6)}]\small{$C_0$}}] at ($(e)!0.3!(f)$) {};
\node [label={[shift={(-0.2,-0.6)}]\small{$C_1$}}] at ($(d)!0.5!(e)$) {};
\node [label={[shift={(-0.3,-0.3)}]\small{$C_2$}}] at ($(c)!0.5!(d)$) {};
\node [label={[shift={(-0.2,-0.2)}]\small{$C_3$}}] at ($(b)!0.7!(c)$) {};
\node [label={[shift={(0,0)}]\small{$C_4$}}] at ($(b)!0.5!(a)$) {};
\node [label={[shift={(0,0)}]\small{$C_5$}}] at ($(i)!0.5!(a)$) {};
\node [label={[shift={(0.2,-0.2)}]\small{$C_6$}}] at ($(h)!0.3!(i)$) {};
\node [label={[shift={(0.3,-0.3)}]\small{$C_7$}}] at ($(g)!0.5!(h)$) {};

\node at ($(a)!1.3!($(e)!0.5!(f)$)$) {$\mathcal{O}$};
\node at ($(g)!1.3!($(b)!0.5!(c)$)$) {$t_1$};
\node at ($(d)!1.3!($(h)!0.5!(i)$)$) {$t_2$};

 \draw ($(a)!-0.3!(b)$) -- ($(a)!1.3!(b)$)[line width =1.2pt,YellowGreen!100];
 \draw ($(b)!-0.3!(c)$) -- ($(b)!1.3!(c)$)[line width =1.2pt,BlueGreen!100];
 \draw ($(c)!-0.3!(d)$) -- ($(c)!1.3!(d)$)[line width =1.2pt,Blue!80];
 \draw ($(d)!-0.3!(e)$) -- ($(d)!1.3!(e)$)[line width =1.2pt,blue!20];
 \draw ($(e)!-0.3!(f)$) -- ($(e)!1.3!(f)$)[line width =1.2pt];
 \draw ($(f)!-0.3!(g)$) -- ($(f)!1.3!(g)$)[line width =1.2pt,blue!20];
 \draw ($(g)!-0.3!(h)$) -- ($(g)!1.3!(h)$)[line width =1.2pt,Blue!80];
 \draw ($(h)!-0.3!(i)$) -- ($(h)!1.3!(i)$)[line width =1.2pt,BlueGreen!100];
 \draw ($(i)!-0.3!(a)$) -- ($(i)!1.3!(a)$)[line width =1.2pt,YellowGreen!100];
 
 \draw ($(a)!0.8!($(e)!0.5!(f)$)$) -- ($(a)!1.2!($(e)!0.5!(f)$)$)[line width =1.2pt,dashed];
 \draw ($(g)!0.8!($(b)!0.5!(c)$)$) -- ($(g)!1.2!($(b)!0.5!(c)$)$)[blue,line width =1.2pt,dashed];
 \draw ($(d)!0.8!($(h)!0.5!(i)$)$) -- ($(d)!1.2!($(h)!0.5!(i)$)$)[blue,line width =1.2pt,dashed];

  \foreach \x /\alph /\name in {270/j/$ $, 310/k/$ $, 350/l/$ $,  30/m/$ $,  70/n/$ $,  110/o/$ $, 150/p/$ $, 190/q/$ $, 230/r/$ $}{
  \node[xshift=7cm,circle,inner sep=0,minimum size=0.0pt] (\alph) at (\x:1.5cm) {\name}; }
\node [label={[shift={(0.2,-0.6)}]\small{$C_8$}}] at ($(o)!0.5!(p)$) {};
\node [label={[shift={(0,-0.6)}]\small{$C_0$}}] at ($(n)!0.5!(o)$) {};
\node [label={[shift={(-0.2,-0.6)}]\small{$C_1$}}] at ($(m)!0.5!(n)$) {};
\node [label={[shift={(-0.3,-0.4)}]\small{$C_2$}}] at ($(l)!0.5!(m)$) {};
\node [label={[shift={(-0.3,-0.2)}]\small{$C_3$}}] at ($(k)!0.5!(l)$) {};
\node [label={[shift={(0,0)}]\small{$C_4$}}] at ($(k)!0.5!(j)$) {};
\node [label={[shift={(0,0)}]\small{$C_5$}}] at ($(r)!0.5!(j)$) {};
\node [label={[shift={(0.3,-0.2)}]\small{$C_6$}}] at ($(q)!0.5!(r)$) {};
\node [label={[shift={(0.3,-0.4)}]\small{$C_7$}}] at ($(p)!0.5!(q)$) {};

 \draw ($(j)!-0.3!(k)$) -- ($(j)!1.3!(k)$)[line width =1.2pt,YellowGreen!100];
 \draw ($(k)!-0.3!(l)$) -- ($(k)!1.3!(l)$)[line width =1.2pt,BlueGreen!100,dashed];
 \draw ($(l)!-0.3!(m)$) -- ($(l)!1.3!(m)$)[line width =1.2pt,Blue!80];
 \draw ($(m)!-0.3!(n)$) -- ($(m)!1.3!(n)$)[line width =1.2pt,blue!20];
 \draw ($(n)!-0.3!(o)$) -- ($(n)!1.3!(o)$)[line width =1.2pt,dashed];
 \draw ($(o)!-0.3!(p)$) -- ($(o)!1.3!(p)$)[line width =1.2pt,blue!20];
 \draw ($(p)!-0.3!(q)$) -- ($(p)!1.3!(q)$)[line width =1.2pt,Blue!80];
 \draw ($(q)!-0.3!(r)$) -- ($(q)!1.3!(r)$)[line width =1.2pt,BlueGreen!100,dashed];
 \draw ($(r)!-0.3!(j)$) -- ($(r)!1.3!(j)$)[line width =1.2pt,YellowGreen!100];
 
  \foreach \x /\alph /\name in {90/s/$ $, 150/t/$ $, 210/u/$ $,  270/v/$ $,  330/w/$ $,  30/x/$ $}{
  \node[xshift=3.5cm,yshift=-4.7cm,circle,inner sep=0,minimum size=0.0pt] (\alph) at (\x:1.5cm) {\name}; }

\node [label={[shift={(0.2,-0.6)}]\small{$C_8$}}] at ($(s)!0.5!(t)$) {};
\node [label={[shift={(-0.2,-0.6)}]\small{$C_1$}}] at ($(x)!0.5!(s)$) {};
\node [label={[shift={(-0.3,-0.3)}]\small{$C_2$}}] at ($(w)!0.5!(x)$) {};
\node [label={[shift={(0,0)}]\small{$C_4$}}] at ($(v)!0.5!(w)$) {};
\node [label={[shift={(0,0)}]\small{$C_5$}}] at ($(u)!0.5!(v)$) {};
\node [label={[shift={(0.3,-0.3)}]\small{$C_7$}}] at ($(t)!0.5!(u)$) {};

 \draw ($(s)!-0.3!(t)$) -- ($(s)!1.3!(t)$)[line width=1.2pt,blue!20,dashed];
 \draw ($(t)!-0.3!(u)$) -- ($(t)!1.3!(u)$)[line width=1.2pt,Blue!80,dashed];
 \draw ($(u)!-0.3!(v)$) -- ($(u)!1.3!(v)$)[line width=1.2pt,YellowGreen!100,dashed];
 \draw ($(v)!-0.3!(w)$) -- ($(v)!1.3!(w)$)[line width=1.2pt,YellowGreen!100,dashed];
 \draw ($(w)!-0.3!(x)$) -- ($(w)!1.3!(x)$)[line width=1.2pt,Blue!80,dashed];
 \draw ($(x)!-0.3!(s)$) -- ($(x)!1.3!(s)$)[line width=1.2pt,Blue!20,dashed]; 
 
\foreach \x /\alph /\name in {90/y/$ $, 180/z/$ $, 270/aa/$ $,  360/bb/$ $}{
  \node[xshift=0cm,yshift=-9.2cm,circle,inner sep=0,minimum size=0.0pt] (\alph) at (\x:1.5cm) {\name}; }
\node [label={[shift={(0.2,-0.6)}]\small{$C_8$}}] at ($(y)!0.5!(z)$) {};
\node [label={[shift={(-0.2,-0.1)}]\small{$0$}}] at ($(y)!0.5!(z)$) {};
\node [label={[shift={(-0.2,-0.6)}]\small{$C_1$}}] at ($(bb)!0.5!(y)$) {};
\node [label={[shift={(0.2,-0.1)}]\small{$0$}}] at ($(y)!0.5!(bb)$) {};
\node [label={[shift={(0.3,-0.5)}]\small{$0$}}] at ($(aa)!0.5!(bb)$) {};
\node [label={[shift={(-0.3,-0.5)}]\small{$0$}}] at ($(z)!0.5!(aa)$) {};
\node [label={[shift={(-0.2,0)}]\small{$C_4$}}] at ($(aa)!0.5!(bb)$) {};
\node [label={[shift={(0.2,0)}]\small{$C_5$}}] at ($(z)!0.5!(aa)$) {};

 \draw ($(y)!-0.3!(z)$) -- ($(y)!1.3!(z)$)[line width=1.2pt,blue!20];
 \draw ($(z)!-0.3!(aa)$) -- ($(z)!1.3!(aa)$)[line width=1.2pt,YellowGreen!100];
 \draw ($(aa)!-0.3!(bb)$) -- ($(aa)!1.3!(bb)$)[line width=1.2pt,YellowGreen!100];
 \draw ($(bb)!-0.3!(y)$) -- ($(bb)!1.3!(y)$)[line width=1.2pt,blue!20];

\foreach \x /\alph /\name in {180/cc/$ $, 300/dd/$ $, 60/ee/$ $}{
  \node[xshift=7cm,yshift=-11cm,circle,inner sep=0,minimum size=0.0pt] (\alph) at (\x:1.5cm) {\name}; }
\node [label={[shift={(0.2,-0.2)}]\small{$C_5$}}] at ($(cc)!0.5!(dd)$) {};
\node [label={[shift={(0,-0.6)}]\small{$0$}}] at ($(cc)!0.5!(dd)$) {};
\node [label={[shift={(0.2,-0.6)}]\small{$C_8$}}] at ($(ee)!0.5!(cc)$) {};
\node [label={[shift={(0,-0.1)}]\small{$0$}}] at ($(cc)!0.5!(ee)$) {};
\node [label={[shift={(0.3,-0.3)}]\small{$0$}}] at ($(dd)!0.5!(ee)$) {};
\node [label={[shift={(-0.3,-0.3)}]\small{$C_2$}}] at ($(dd)!0.5!(ee)$) {};

 \draw ($(cc)!-0.3!(dd)$) -- ($(cc)!1.3!(dd)$)[line width=1.2pt,YellowGreen!100];
 \draw ($(dd)!-0.3!(ee)$) -- ($(dd)!1.3!(ee)$)[line width=1.2pt,Blue!80];
 \draw ($(ee)!-0.3!(cc)$) -- ($(ee)!1.3!(cc)$)[line width=1.2pt,blue!20];  
 
\draw (2.5,0) -- (4.5,0) [->];
\node [label={[shift={(0,0)}]$\beta_1$}] at (3.5,0){};
\draw (5.7,-1.66) -- (4.7,-2.95) [->];
\node [label={[shift={(-0.2,0)}]$\beta_2$}] at (5.2,-2.31){};
\draw (2.3,-6.59) -- (1.3,-7.88) [->];
\node [label={[shift={(-0.2,0)}]$\beta_3$}] at (1.8,-7.24){};
\draw (4.7,-6.59) -- (6.2,-9.7) [->];
\node [label={[shift={(2,0)}]only if $G_{R_9}=\{id\}$}] at (5,-8){};

\node [label={[shift={(0,0)}]4 lines in $\mathbb{P}^1\times\mathbb{P}^1$}] at (0,-12.5){};
\node [label={[shift={(0,0)}]3 lines in $\mathbb{P}^2$}] at (7,-14){};
\end{tikzpicture}
\caption{Two ways to contract the fiber $I_9$}
		\label{blowdown I9}
\end{figure}
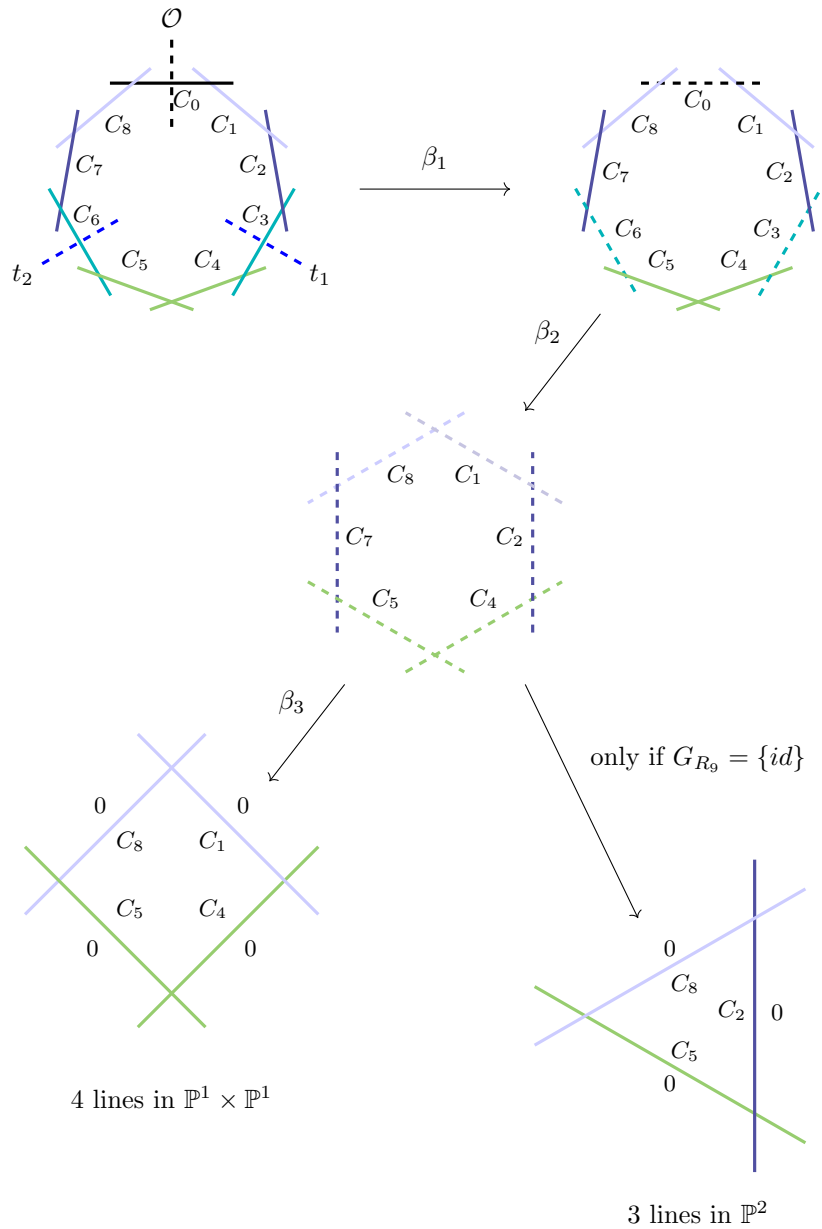
\end{center}

\begin{rem}{\rm If one contracts $R_9$ to $\mathbb{P}^1\times\mathbb{P}^1$, the elliptic involution defined on $R_9$ induces an involution of $\mathbb{P}^1\times\mathbb{P}^1$, which is precisely the exchange of the two rulings. Indeed $\iota_{R_9}$ maps $C_1$ to $C_8$ and $C_4$ to $C_5$, so the automorphism induced by $\iota_{R_9}$ on $\mathbb{P}^1\times\mathbb{P}^1$ maps the $(1,0)$-curves $\beta_3(\beta_2(\beta_1(C_1)))$ and $\beta_3(\beta_2(\beta_1(C_5)))$ to the $(0,1)$-curves $\beta_3(\beta_2(\beta_1(C_8)))$ and $\beta_3(\beta_2(\beta_1(C_4)))$. }\end{rem}

\begin{rem}\label{twomodelsR9}{\rm Over the complex field all the rational elliptic fibrations can be contracted to a pencil of cubics in $\mathbb{P}^2$ and for each extremal rational elliptic fibration, the equation of an associated pencil is known \cite[Proof of Theorem 5.6.2]{CoDol}. In particular, an equation of a pencil of cubics associated to the surface $R_9$ in $\mathbb{P}^2_{(z_0:z_1:z_2)}$ is given by $$\mathcal{P}_9:=(z_0z_1z_2)+t(z_0^2z_1+z_1^2z_2+z_2^2z_0).$$ 
		 The base points of $\mathcal{P}_9$ are $(1:0:0)$, $(0:1:0)$ and $(0:0:1)$, each with multiplicity~3. After blowing up these points one obtains a rational elliptic surface, with a reducible fiber over $t=0$.
	A birational map from $\mathbb{P}^2$ to $\mathbb{P}^1\times \mathbb{P}^1$ is given by the blow-up of two points and the contraction of the line through these points. For example the maps 
	$$\alpha_1:\mathbb{P}^1_{(x_0:x_1)}\times\mathbb{P}^1_{(y_0:y_1)}\ra\mathbb{P}^2, \ \alpha_1((x_0:x_1),(y_0:y_1))=(x_0y_0:x_0y_1:x_1y_0)\mbox{ and }$$
	
	$$\alpha_2:\mathbb{P}^2_{(z_0:z_1:z_2)}\ra\mathbb{P}^1\times\mathbb{P}^1,\   \alpha_2(z_0:z_1:z_2)=((z_0:z_2),(z_0:z_1))$$ 
	are birational inverse maps. They correspond to blowing up the points $(0:1:0)$ and $(0:0:1)$ in $\mathbb{P}^2$ and to contracting the line $z_0=0$. 
	
	We observe that the points $(0:1:0)$ and $(0:0:1)$ are base points of the pencil $\mathcal{P}_9$. The birational image of this pencil is a bidegree $(2,2)$ pencil in $\mathbb{P}^1\times\mathbb{P}^1$, given by
	\begin{equation}\label{pencil in P1xP1 R9} (x_0x_1y_0y_1)+t(x_0^2y_0y_1+x_0x_1y_1^2+x_1^2y_0^2),\end{equation} which still corresponds to $R_9$.
We conclude that the pencil \eqref{pencil in P1xP1 R9}, considered as pencil of curves of bidegree $(2,2)$ over a certain field $k$, defines the rational elliptic surface $R_9$ over $k$. We already observed that the Galois action on $R_9$ corresponds to an involution exchanging the rulings of $\mathbb{P}^1\times\mathbb{P}^1$ and indeed, with the chosen coordinates, it is  $((x_0:x_1),(y_0:y_1))\mapsto ((y_1:y_0),(x_1:x_0))$.}\end{rem}

The following example illustrates the two different Galois actions that occur in Proposition \ref{prop: Galois of R9}.

\begin{example}
In Remark \ref{twomodelsR9}, we saw that the pencil of cubics given by $\mathcal{P}_9=(z_0z_1z_2) + t(z^2_0z_1 + z^2_1z_2 + z^2_2z_0)$ gives rise to an $R_9$ surface. A Weierstrass equation for this surface is $$y^2 = x^3 -(432t^3 + 10368)xt + 3456t^6 +  124416t^3 + 746496,$$ and the Mordell--Weil group consists of three sections defined over $\mathbb{Q}$, which are given by $[0,1,0]$ and $[12t^2,\pm 864,1]$. We conclude from \ref{prop: Galois of R9} that in this case we have $G_{R_9}=\{id\}$.\\
Another example of an $R_9$ surface is given by the Weierstrass equation $$y^2=x^3-3(t^3+24)xt+2(t^6+36t^3+216),$$ which  has Mordell--Weil group given by the section $[0,1,0]$ and the two sections $[t^2-1,\pm3\sqrt{3}t,1]$ \cite[Table 5.3]{MiPe}. So the Mordell--Weil group of this surface is trivial over $\mathbb{Q}$, and defined over the quadratic extension $\mathbb{Q}(\sqrt{3})$. We conclude from Proposition \ref{prop: Galois of R9} that in this case we have $G_{R_9}=\mathbb{Z}/2\mathbb{Z}$, and the surface can not be contracted to $\mathbb{P}^2$.
\end{example}

\subsection{The K3 surface \texorpdfstring{$X_9$}{X9}}\hspace{0pt}

\vspace{5pt}

Let $X_9$ be a K3 surface obtained by a generic base change of order 2 on the rational elliptic surface $R_9$  as described in Section \ref{sec: setting}.  

Then the elliptic fibration $\mathcal{E}_{R_9}:R_9\ra\mathbb{P}^1$ induces an elliptic fibration $\mathcal{E}_{X_9}:X_9\ra\mathbb{P}^1$ on $X_9$. We denote by $\iota_{X_9}$ the elliptic involution on $\mathcal{E}_{X_9}$. We denote by $\tau_{9}$ the cover involution of $\pi:X_9\ra R_9$.

By definition the fibration $\mathcal{E}_9$ is of type 2 with respect to $\tau$. So, by Theorem \ref{theo: types and field of definition}, the field of definition of the elliptic fibration and of a section of it is $k$.

Nevertheless there could be other sections or components of some reducible fibers which are not defined over $k$. 

In what follows we describe the Galois group $G_{\mathcal{E}_{X_9}, \mathrm{MW}}$, i.e., the Galois group 
of the field extension $k_{\mathcal{E}_{X_9}}/k$ 
over which all components of reducible fibers and sections of the fibration $\mathcal{E}_{X_9}$ are defined. 

\begin{proposition}\label{prop: GaloisXi}
	The Galois group $G_{\mathcal{E}_{X_9}, \mathrm{MW}}$ of the elliptic fibration $\mathcal{E}_{X_9}:X_9\ra\mathbb{P}^1$ is as follows
	
	More precisely, \begin{itemize} 
	\item $G_{\mathcal{E}_{X_9}, \mathrm{MW}}\simeq\left(\Z/2\Z\right)^2$ if, and only if, $G_{R_9}\simeq\Z/2\Z$ and the branch fibers of $\pi$ are not defined over $k_{R_9}$,
	\item $G_{\mathcal{E}_{X_9}, \mathrm{MW}}\simeq\Z/2\Z$ if, and only if, $G_{R_9}\simeq\Z/2\Z$ and the branch fibers of $\pi$ are defined over $k_{R_9}$,
		\item $G_{\mathcal{E}_{X_9}, \mathrm{MW}}\simeq\Z/2\Z$ if, and only if, $G_{R_9}=\{ id \}$ and the branch fibers of $\pi$ are not defined over $k_{R_9}$,
		\item	$G_{\mathcal{E}_{X_9}, \mathrm{MW}}=\{ id \}$ if, and only if, $G_{R_9}=\{id\}$ and the branch fibers of $\pi$ are defined over $k_{R_9}$.
\end{itemize}\end{proposition}
\begin{proof}
 This follows from a simple analysis of whether the extensions $k_{R_9}$ and $k_{\tau_9}$ are linearly disjoint or not. This depends of course on the branch locus of the base change map. See the discussion in Notation \ref{not: compositum}.
\end{proof}

The elliptic fibration $\mathcal{E}_{X_9}$ has two fibers of type $I_9$. Let us denote by $\Theta_i^{j}$, for $i=0,\ldots, 8$, $j=1,2$ the $i$-th component of the $j$-th fiber of type $I_9$. The sections of $\mathcal{E}_{R_9}$ induce sections of $\mathcal{E}_{X_9}$, and thus $\mathrm{MW}(\mathcal{E}_{X_9})=\{\mathcal{O}_{X_9}, T_1, T_2\}$. Thus $$\pi(\mathcal{O}_{X_9})=\mathcal{O},\ \ \pi(T_1)=t_1,\ \ \pi(T_2)=t_2,\ \ \pi(\Theta_i^{j})=C_i,\ i=0,\ldots, 8,\ j=1,2.$$
The automorphism $\tau_9$ is the cover involution of $\pi$ and thus $$\tau_9(\mathcal{O}_{X_9})=\mathcal{O}_{X_9},\ \   \tau_9(T_1)=T_1,\ \ \tau_9(T_2)=T_2, \ \ \tau_9(\Theta_i^{1})=\Theta_i^{2}, \, i=0,\ldots, 8.$$
Figure \ref{diagram X9} summarizes the above.  

\begin{center}
\begin{figure}[h]
 \begin{tikzpicture}[scale=1.2] 
\node [label= above right:$\mathcal{O}_{X_9}$, draw,circle,inner sep=0pt,minimum size=4pt] at (2.5,-0.5) {};
 \node [label=above:$T_2$, draw,circle,inner sep=0pt,minimum size=4pt] at (-1.5,-0.5) {};
 \node [label=above:$T_1$, draw,circle,inner sep=0pt,minimum size=4pt] at (6.5,-0.5) {};
 \node [label=below:$\Theta_0^2$,draw,circle,inner sep=0pt,minimum size=4pt] at (2.5,-1.5) {};
 \node [label=left:$\Theta_1^2$,draw,circle,inner sep=0pt,minimum size=4pt] at (4,-2) {};
 \node [label=left:$\Theta_2^2$,draw,circle,inner sep=0pt,minimum size=4pt] at (4,-3) {};
 \node [label=above left:$\Theta_3^2$,draw,circle,inner sep=0pt,minimum size=4pt] at (4,-4) {};
\node [label=above:$\Theta_4^2$,draw,circle,inner sep=0pt,minimum size=4pt] at (3,-4) {};
 \node [label=above:$\Theta_5^2$,draw,circle,inner sep=0pt,minimum size=4pt] at (2,-4) {};
 \node [label=above right:$\Theta_6^2$,draw,circle,inner sep=0pt,minimum size=4pt] at (1,-4) {};
 \node [label=right:$\Theta_7^2$,draw,circle,inner sep=0pt,minimum size=4pt] at (1,-3) {};
 \node [label=right:$\Theta_8^2$,draw,circle,inner sep=0pt,minimum size=4pt] at (1,-2) {};
 \node [label=above:$\Theta_0^1$,draw,circle,inner sep=0pt,minimum size=4pt] at (2.5,0.5) {};
 \node [label=left:$\Theta_1^1$,draw,circle,inner sep=0pt,minimum size=4pt] at (4,1) {};
 \node [label=left:$\Theta_2^1$,draw,circle,inner sep=0pt,minimum size=4pt] at (4,2) {};
 \node [label=below left:$\Theta_3^1$,draw,circle,inner sep=0pt,minimum size=4pt] at (4,3) {};
\node [label=below:$\Theta_4^1$,draw,circle,inner sep=0pt,minimum size=4pt] at (3,3) {};
 \node [label=below:$\Theta_5^1$,draw,circle,inner sep=0pt,minimum size=4pt] at (2,3) {};
 \node [label=below right:$\Theta_6^1$,draw,circle,inner sep=0pt,minimum size=4pt] at (1,3) {};
 \node [label=right:$\Theta_7^1$,draw,circle,inner sep=0pt,minimum size=4pt] at (1,2) {};
 \node [label=right:$\Theta_8^1$,draw,circle,inner sep=0pt,minimum size=4pt] at (1,1) {};

\path[every node/.style={font=\sffamily\small}]
     (2.5,-0.5) edge (2.5,-1.5)
     (2.5,-0.5) edge (2.5,0.5)
     (2.5,-1.5) edge (4,-2)
     (2.5,-1.5) edge (1,-2)
     (2.5,0.5)  edge (4,1)
     (2.5,0.5)  edge (1,1)
     (4,-2)		edge (4,-3)
     (4,-3)		edge (4,-4)
     (4,-4)		edge (3,-4)
     (3,-4)		edge (2,-4)
     (2,-4)		edge (1,-4)
     (1,-4)		edge (1,-3)
     (1,-3)		edge (1,-2)
     (4,1)		edge (4,2)
     (4,2)		edge (4,3)
     (4,3)		edge (3,3)
     (3,3)		edge (2,3)
     (2,3)		edge (1,3)
     (1,3)		edge (1,2)
     (1,2)		edge (1,1)
     (-1.5,-0.5)edge (1,3)
     (-1.5,-0.5)edge (1,-4)
     (6.5,-0.5)	edge (4,3) 
     (6.5,-0.5)	edge (4,-4) 
     ;
\draw (-2,-0.5) -- (7,-0.5)[dashed];

\path[<->,>=stealth',auto,every node/.style={font=\sffamily\small}](7.2,-1) edge[bend right] node [right]{$\tau_9$} (7.2,0);
\end{tikzpicture}
\caption{Reducible fibers and sections of the fibration $\mathcal{E}_{X_9}$ on $X_9$}
		\label{diagram X9}
\end{figure}
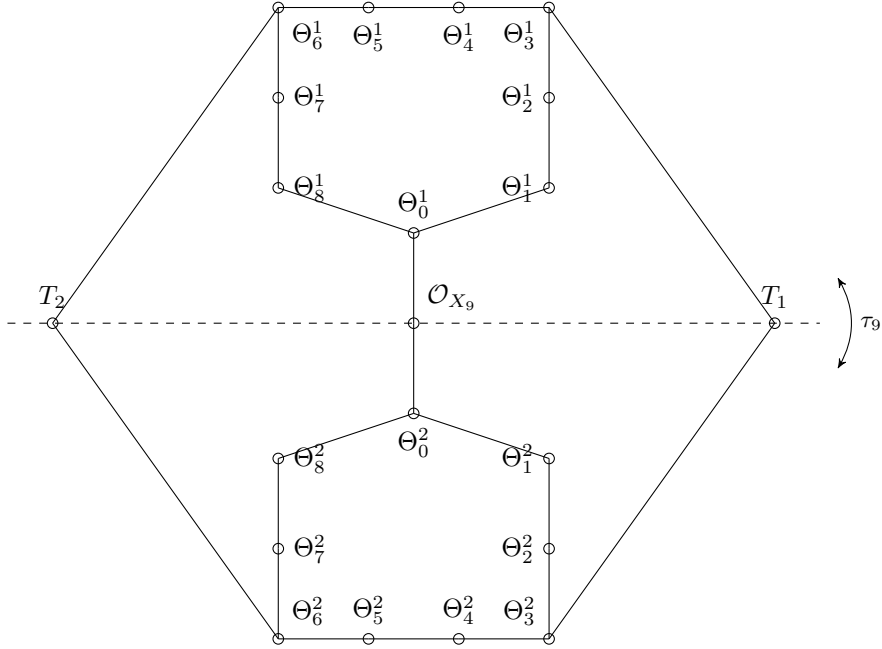 
\end{center} 
 
\begin{proposition}\label{latticeA8} The N\'eron--Severi group of $X_9$ has rank 18, signature $(1,17)$, discriminant group $\Z/9\Z$ and discriminant form is the opposite to the one of $A_8$. The transcendental lattice of $X_9$ is the unique (up to isometries) even lattice with signature $(2,2)$, discriminant group $\Z/9\Z$ and discriminant form equal to the one of $A_8$.\end{proposition}
\begin{proof}
The N\'eron--Severi group contains the 18 linearly independent classes $\mathcal{O}_{X_9}$, $T_1$, $T_2$ and $\Theta_i^{j}$, for $i=1,\ldots, 8$, $j=1,2$. Hence it has rank at least 18. On the other hand the family of $X_9$ is a two dimensional family (because of the choice of two branch fibers of the double cover $X_9\ra R_9$). So the N\'eron--Severi has rank at most 18. We conclude that the 18 classes listed before form a basis of $\mathrm{NS}(X_9)$.The intersection form and the discriminant form of $\mathrm{NS}(X_9)$ can be explicitly computed and one can check that it has discriminant 9. In particular, a generator for the discriminant group is $2\left(\sum_{i=1}^9i\left(\Theta_i^{1}-\Theta_i^{2}\right)\right)/9$ and its discriminant form is $\Z/9\Z\left(\frac{8}{9}\right)$, which is the opposite to the discriminant form of $A_8$. The discriminant form of the transcendental lattice is the opposite of the discriminant form of the N\'eron--Severi group. Hence the transcendental lattice $T_{X_9}$ is an even lattice with signature $(2,2)$ and discriminant form $\Z/9\Z\left(\frac{-8}{9}\right)$. The transcendental lattice is uniquely determined by these data by \cite[Theorem 1.13.2]{NikulinIntQuadForms}. We observe that the discriminant form of $T_{X_9}$ is the same as the one of $A_8$ and that $\rk(T_{X_9})+4=\rk(A_8)$.\end{proof}

\begin{corollary}\label{cor: kE9=kRt} The filed $k_{\mathcal{E}_9}$ coincides with $k_{R,\tau}$.\end{corollary}
	\begin{proof}
		By Proposition \ref{latticeA8} the classes of the reducible fibers and of the sections of $\mathcal{E}_{X_9}$ form a basis of $\mathrm{NS}(X)$. Each of these classes corresponds to a unique curve (since these are negative curves), which is a smooth rational curve. Hence the field where all these classes are defined coincides with the field where $\mathrm{NS}(X)$ is defined. The former is $k_{\mathcal{E}_9}$ by definition, the latter is $k_{R_9,\tau_9}$ by Lemma \ref{generators_NS(X)}.  
\end{proof}

\subsection{Classification of all the possible fibrations of the \texorpdfstring{$K3$}{K3} surface \texorpdfstring{$X_9$}{X9}}\label{sec: classification of the fibrations of X9}\hspace{0pt}

\vspace{5pt}

In order to find all elliptic fibrations on $X_9$, we use Nishiyama's method explained in \cite{Nish}. As explained in \cite[Section 6.1]{Nish}, if one is able to find a lattice $T_0$ which is negative definite, has the same discriminant form of the transcendental lattice of a K3 surface and its rank is the rank of the transcendental group plus four, then there is an operative method to classify the configuration of the reducible fibers of the elliptic fibrations on the surface. In our particular case, by Proposition \ref{latticeA8}, we put $T_0=A_8$ and in  order to classify the elliptic fibrations on $X_9$ (and in particular the lattice $W$ of each of these elliptic fibration, with the notation of \cite{Nish}) we have to find the orthogonal complements of primitive embeddings of the root lattice $A_8$ in the 24 possible lattices listed (by their root type) by Niemeier \cite[Satz 8.3]{Nie} (or \cite[Theorem 1.7]{Nish}). By \cite[Lemmas 4.1 and 4.3]{Nish} we know that $A_8$ embeds primitively uniquely, up to the action of the Weyl group, in $A_m$ for $m\geq 8$, in $D_n$ for $n\geq 9$, and in no other root lattice. The orthogonal complements of these embeddings in the 24 Niemeier lattices are then found in \cite[Corollary 4.4]{Nish}, and this determines the reducible fibers and the rank of the Mordell--Weil group for each fibration. These results are summarized in Table \ref{eq: table elliptic fibrations on X9}. Note that line 1 is the fibration $\mathcal{E}_{X_9}$.\\ Apart from the torsion part of the Mordell--Weil group, everything is found by Nishiyama's method as explained above. We compute the torsion parts in what follows.

\subsubsection{Torsion of the Mordell--Weil group for the elliptic fibrations associated to \texorpdfstring{$X_9$}{X9}}\hspace{0pt}

\vspace{5pt}

By \cite[Table 1]{Shimada}, we can immediately conclude that the torsion of the fibrations in lines 2, 3, 4, 5, 8, 9, and 12 is trivial, and the torsion part of fibrations 6, 7, 10, and 11 is either $\mathbb{Z}/2\mathbb{Z}$ or trivial.

Fibration 11 comes from the orthogonal complement of the embedding of $A_8$ in a lattice $N$ of rank 24 with root type $A_{24}$. We observe that  $N/A_{24}=\mathbb{Z}/5\mathbb{Z}$ (\cite[Satz 8.3]{Nie} or \cite[Theorem 1.7]{Nish}). By \cite[Lemma 6.6, $iii)$]{Nish}, the torsion of the elliptic fibration corresponding to this embedding of $A_8$ in $N$ has to be contained in $N/A_{24}$, so this fibration does not have a 2-torsion section and the torsion part of the Mordell--Weil group is trivial. 

Note that, in terms of the notation of our configuration of $2I_9$ (see Figure~\ref{diagram X9}), we find a fiber of type $I_{16}$ composed of the following curves on $X_9$. 
$$\Theta_0^1,\Theta_1^1,\Theta_2^1,\Theta_3^1,T_1,\Theta_3^2,\Theta_2^1,\Theta_1^2,\Theta_0^2,\Theta_8^2,\Theta_7^2,\Theta_6^2,T_2,\Theta_6^1,\Theta_7^1,\Theta_8^1.$$
Moreover, $\Theta_5^1,\Theta_4^1,\Theta_5^2,\Theta_4^2$ are sections for this fibration. Let $\Theta_4^1$ be the 0-section, then the height $h(\Theta_5^2)$ of the section $\Theta_5^2$ is $2\cdot2+0-\frac{8(16-8)}{16}=0$ \cite[Chap. 11 \S 11.8]{SS}, and therefore it is a torsion section \cite[Theorem 11.5]{SS}. Since we know that the fibration in line 11 has trivial torsion, and the fibration in line 7 is the only other one with reducible fiber of type $I_{16}$, we conclude that we found a representation of the fibration in line 7, and therefore the torsion part of the Mordell--Weil group of this fibration is $\mathbb{Z}/2\mathbb{Z}$.

Finally, we find that the torsion part of the Mordel--Weil groups of the fibrations in lines 6 and 10 are $\mathbb{Z}/2\mathbb{Z}$ in the same way as we did for line 7.  

We find the class of the fiber of the other elliptic fibrations, by giving the components of one reducible fiber in terms of the configuration of $2I_9$ (see Figure \ref{diagram X9}). The rational curves orthogonal to the class of the fiber are necessarly components of other reducible fibers, hence we list all the irreducible components of at least one reducible fiber and some components of the others fo each elliptic fibration.

\noindent$\underline{I_8^*+I_4}$\\
$\Theta_1^2,\Theta_8^2,2\Theta_0^2,2\mathcal{O}_{X_9},2\Theta_0^1,2\Theta_8^1,2\Theta_7^1,2\Theta_6^1,2\Theta_5^1,2\Theta_4^1,2\Theta_3^1,\Theta_2^1,T_1\;\;+\;\;\Theta_6^2,\Theta_5^2,\Theta_4^2,$\\

\noindent$\underline{I_2^*+I_{10}}$\\
$\Theta_8^1,\Theta_0^1,2\Theta_1^1,2\mathcal{O}_{X_9},2\Theta_0^2, \Theta_8^2,\Theta_1^2\;\;+\;\; T_2,\Theta_6^1,\Theta_5^1,\Theta_4^1,\Theta_3^1,T_1\Theta_3^2,\Theta_4^2,\Theta_5^2,\Theta_6^2,\Theta_6^1.$\\

For this configuration of $I_8^*+I_4$ we find the three sections $\Theta_7^2,\Theta_3^2$, and $\Theta_2^2$. If we set $\Theta_7^2$ as the 0-section, then $\Theta_3^2$ has height 0 and hence it is a $2$-torsion section. Since there is only one fibration with reducible fiber $I_8^*+I_4$ in our list, we conclude that this configuration represents the fibration in line 6. Hence the torsion part of the Mordell--Weil group is $\mathbb{Z}/2\mathbb{Z}$. For the fibration in line 10 we have the same reasoning, after finding the sections $\Theta_4^1,\Theta_7^2,\Theta_6^2,\Theta_2^2$, setting $\Theta_2^1$ as the 0-section and finding that $\Theta_7^2$ is a $2$-torsion section. 

\begin{table}[h]
	\begin{adjustwidth}{+0cm}{}		
		\begin{tabular}{c|c|c|c|c|c}
			$n^o$&Niemeier&embedding&roots orth.&reducible fibers&MW\\
			1& $A_8^{\oplus 3}$ & $A_8\subset A_8$ &  $A_8^{\oplus 2}$ &$2I_9 $& $\mathbb{Z}/3\mathbb{Z}$\\
			2 & $E_8 \oplus D_{16}$ & $A_8 \subset D_{16}$ & $E_8\oplus D_7$ & $II^*+I_3^*$ & $\mathbb{Z}$     \\ 
			3 &$E_7^{\oplus 2}\oplus D_{10}$ & $A_8\subset D_{10}$ & $E_7^{\oplus 2}$ & $2III^* $&$ \mathbb{Z}^2$\\ 
			4 &$E_7\oplus A_{17} $ & $A_8\subset A_{17} $&$E_7\oplus A_8 $&$ III^*+I_9$ & $\mathbb{Z}$\\ 
			5&$D_{24} $ &$ A_8\subset D_{24} $&$ D_{15} $&$ I_{11}^* $&$\mathbb{Z}$  \\ 
			6&$D_{12}^{\oplus 2} $&$ A_8\subset D_{12} $&$ D_{12}\oplus A_3  $&$ I_8^*+I_4$&$\Z/2\Z \oplus\mathbb{Z}$\\ 
			7&$D_9\oplus A_{15}$ & $A_8\subset D_9$ & $A_{15}$& $I_{16}$&$\mathbb{Z}/2\mathbb{Z}\oplus\mathbb{Z}$ \\
			8&$D_9\oplus A_{15}$ & $A_8\subset A_{15}$ & $D_9\oplus A_{6}$& $I_5^*+I_7$&$\mathbb{Z}$ \\
			9&$E_6\oplus D_7\oplus A_{11}$ & $A_8\subset A_{11}$ &$E_6\oplus D_7 \oplus A_2$ & $IV^*+I_3^*+I_3$ &$\mathbb{Z}$\\
			10&$D_6\oplus A_{9}^{\oplus 2}$ & $A_8\subset A_9$ & $D_6\oplus A_9 $&$I_2^*+I_{10} $&$\mathbb{Z}/2\mathbb{Z}\oplus \mathbb{Z}$ \\ 
			11&$A_{24} $&$ A_8\subset A_{24} $&$ A_{15} $&$ I_{16} $&$ \mathbb{Z}$ \\
			12&$A_{12}^{\oplus 2}$ & $A_8\subset A_{12}$&$ A_{12}\oplus A_3 $&$ I_{13}+I_4$&$ \mathbb{Z} $\\ 
		\end{tabular}
		\caption{Elliptic fibrations of $X_9$}
		\label{eq: table elliptic fibrations on X9} 
	\end{adjustwidth}
\end{table}

\subsection{Determining the type of each fibration of \texorpdfstring{$X_9$}{X9}}\label{sec:type fibrations X9}\hspace{0pt}

\vspace{5pt}

In what follows, we assume that the surface $R_9$ is general, i.e., its Galois group $G_{R_9}$ is not trivial. The goal of this section is to find an example, for each fibration $\eta$ in Table \ref{eq: table elliptic fibrations on X9}, and to determine for each example the following: \begin{itemize}
\item[a)] The type with respect to the cover involution $\tau_9$;
\item[b)] an upper bound for the degree over $k$ of a field of definition of the fibration, that is, a field over which the reducible fiber and a 0-section are defined; 
\item[c)] an upper bound for the degree over $k$ of a field $k_{\eta, \mathrm{MW}}$ over which the Mordell--Weil group of the fibration admits a set of generators.
\end{itemize}
\vspace{5pt}

The results are summarized in Table \ref{eq: tableof types of the elliptic fibrations on X9}, with the notations introduced in Sections~\ref{sec: extremal RES}, \ref{sec: double covers on ext. RES}~and~\ref{sec: surfaces R9 and X9}.\\

For each fibration in Table \ref{eq: table elliptic fibrations on X9} we find a configuration of (parts of the) reducible fibers in terms of the curves in Figure \ref{diagram X9}. Note that for lines 6, 7, and 10, this is done in the previous section.

\noindent$\underline{II^*+I_3^*}$\\
$\Theta_6^2,2\Theta_7^2,3\Theta_8^2,4\Theta_0^2,5\mathcal{O}_{X_9},6\Theta_0^1,4\Theta_8^1,2\Theta_7^1,3\Theta_1^1\;\;+\;\;\Theta_4^2,\Theta_2^2,2\Theta_3^2,2T_1,2\Theta_3^1,2\Theta_4^1,\Theta_5^1$.

\vspace{5pt}

\noindent$\underline{2III^*}$\\
$\Theta_5^1,2\Theta_6^1,3T_2,4\Theta_6^2,3\Theta_5^2,2\Theta_4^2,\Theta_3^2,2\Theta_7^2\;\;+\;\;\Theta_1^2,2\Theta_0^2,3\mathcal{O}_{X_9},4\Theta_0^1,3\Theta_1^1,2\Theta_2^1,\Theta_3^1,2\Theta_8^1$.

\vspace{5pt}
\noindent$\underline{III^*+I_9}$\\
$\Theta_0^1,2\Theta_8^1,3\Theta_7^1,4\Theta_6^1,3\Theta_5^1,2\Theta_4^1,\Theta_3^1,2T_2\;\;+\;\; \Theta_7^2,\Theta_8^2,\Theta_0^2,\Theta_1^2,\Theta_2^2,\Theta_3^2,\Theta_4^2,\Theta_5^2$.

\vspace{5pt}

\noindent$\underline{I_{11}^*}$\\
$\Theta_8^2,\Theta_1^2,2\Theta_0^2,2\mathcal{O}_{X_9},2\Theta_0^1, 2\Theta_8^1, 2\Theta_7^1, 2 \Theta_6^1, 2\Theta_5^1, 2\Theta_4^1, 2\Theta_3^1, \Theta_2^1, T_1$.

\vspace{5pt}

\noindent$\underline{I_5^*+I_7}$\\
$\Theta_5^1,T_2, 2\Theta_6^1, 2\Theta_7^1, 2\Theta_8^1, 2\Theta_0^1, 2\mathcal{O}_{X_9}, 2\Theta_0^2, \Theta_1^2, \Theta_8^2\;\;+\;\;\Theta_5^2, \Theta_4^2, \Theta_3^2, T_1, \Theta_3^1, \Theta_2^1$.

\vspace{5pt}

\noindent$\underline{IV^*+I_3^*+I_3}$\\
$\Theta_6^2, 2T_2, 3\Theta_6^1, 2\Theta_5^1, \Theta_4^1, 2\Theta_7^1, \Theta_8^1\;\;+\;\;\mathcal{O}_{X_9},\Theta_8^2,2\Theta_0^2, 2\Theta_1^2, 2\Theta_2^2, 2\Theta_3^2, T_1, \Theta_4^2\;\;+\;\;\Theta_1^1, \Theta_2^1$.

\vspace{5pt}

\noindent$\underline{I_{13}+I_4}$\\
$T_2, \Theta_6^1, \Theta_5^1, \Theta_4^1, \Theta_3^1, T_1, \Theta_3^2, \Theta_2^2, \Theta_1^2, \Theta_0^2, \Theta_8^2, \Theta_7^2, \Theta_6^2\;\;+\;\;\Theta_8^1, \Theta_0^1, \Theta_1^1$.\\

\vspace{5pt}

To find the elliptic fibration described in line 11 of the Table \ref{eq: table elliptic fibrations on X9}, we need to find another rational curve on $X_9$. We recall that the K3 surface $X_9$ has an infinite number of rational curves, and considering the elliptic fibration with fibers $IV^*+I_3^*+I_3$ we are able to describe one of them. Indeed, the divisor $$F:=\Theta_6^2+ 2T_2+ 3\Theta_6^1+ 2\Theta_5^1+ \Theta_4^1+ 2\Theta_7^1+ \Theta_8^1$$ corresponds to a reducible fiber of type $IV^*$ of this elliptic fibration. In particular $F$ is the class of the fiber of this fibration, and the divisor $\mathcal{O}_{X_9}+\Theta_8^2+2\Theta_0^2+2\Theta_1^2+2\Theta_2^2+2\Theta_3^2+T_1+ \Theta_4^2$ is linearly equivalent to $F$ and corresponds to the fiber of type $I_3^*$. The remaining reducible fiber consists of three curves meeting in a triangle, one is $\Theta_1^1$, one is $\Theta_2^1$ and we denote the third one by $M$. Since $\Theta_1^1+\Theta_2^1+M$ is a fiber of the elliptic fibration, $M$ is linearly equivalent to $F-\Theta_1^1-\Theta_2^1$. In particular this implies that the intersections properties of $M$ are the following: $M\Theta_1^1=M\Theta_2^1=M\Theta_7^2=M\Theta_5^2=1$ and $M$ is orthogonal to all the other curves appearing in Figure \ref{diagram X9}. 

Let us consider the following configuration of curves:

\vspace{5pt}

\noindent$\underline{I_{16}}$\\
$\Theta_6^1, \Theta_5^1, \Theta_4^1, \Theta_3^1, \Theta_2^1, M, \Theta_5^2, \Theta_4^2, \Theta_3^2, \Theta_2^2, \Theta_1^2, \Theta_0^2, \mathcal{O}_{X_9}, \Theta_0^1, \Theta_8^1, \Theta_7^1.$
\vspace{5pt} 

The curves $T_2$ and $\Theta_7^2$ are sections of this fibration. Assume that $T_2$ is the zero section, then $\Theta_7^2$ is a section, orthogonal to the zero section and meeting the reducible fibers $I_{16}$ in his fifth component. Thus the height of this section is $\frac{9}{16}$. As a consequence, the lattices spanned by the irreducible components of the reducible fiber of type $I_{16}$, the zero section $T_2$ and the section $\Theta_7^2$ is a sublattice of $\mathrm{NS}(X_9)$, which has the same rank and the same discriminant of $\mathrm{NS}(X_9)$ and therefore coincides with $\mathrm{NS}(X_9)$. So there are no torsion sections for this elliptic fibration (otherwise one should add their contribution to obtain the N\'eron--Severi group). As a consequence the fibration whose class of the fiber is $\Theta_6^1+ \Theta_5^1+\Theta_4^1+\Theta_3^1+\Theta_2^1+M+\Theta_5^2+\Theta_4^2+\Theta_3^2+\Theta_2^2+\Theta_1^2+\Theta_0^2+\mathcal{O}_{X_9}+\Theta_0^1+\Theta_8^1+\Theta_7^1$ corresponds to the fibration in line 11 of Table \ref{eq: table elliptic fibrations on X9}.

	\begin{corollary} For each fibration in Table \ref{eq: table elliptic fibrations on X9}, there exists at least one elliptic fibration on $X_9$ with the properties given in the list which is defined over $k_{R_9,\tau_9}$.
\end{corollary}
\begin{proof} The result follows by \ref{theo: types and field of definition} for the fibration of type 1 and 2. For the fibration of type 3, one wants to apply the Remark \ref{lemma: type 3}. For all the listed fibration with the exception of the 11, we are able to write the class of the fiber as a linear combinantion of $\Theta_i^j$, $\mathcal{O}_{X_9}$ and $T_k$. All these curves are defined on $k_{R_9,\tau_9}$, by \ref{cor: kE9=kRt}. In the case of the fibration 11, we introduced another curve, $M$. Since its class is written as linear combination of the classes generation $\mathrm{NS}(X)$, its class is defined on $k_{R_9,\tau_8}$. Since it negative effective class, we deduce that it is supported either on an irreducible rational curve or on the union of rational curves. Since it is a component of a fiber of a certain fibration, at least on the closure of the field of definition of the fibration, it is an irreducbile curve (where it is defined). Hence, $M$ is defined an irreducible smooth rational curve defined over $k_{R_9,\tau_9}$ 
\end{proof}

We gave an example for each fibrations in Table \ref{eq: table elliptic fibrations on X9}. We choose a section for each of them to be the zero section and we determine their type with respect to $\tau$. By using Proposition \ref{prop: GaloisXi} we describe the properties of the fields $k_{\eta,MW}$, which follows by the previous Corollary. The results are listed in Table \ref{eq: tableof types of the elliptic fibrations on X9}.

\begin{table}[h!]
\begin{tabular}{c|c|c|c|c|c|c}
$\mbox{n}^o$&roots orth.&type&sections&\makecell{field of \\def. \\0-section}&\makecell{field of\\def. all\\sections}&$[k_{\eta,\mathrm{MW}}:k]$\\
1 & $A_8^{\oplus 2}$ & 2 & $\mathcal{O}_{X_9},T_1,T_2$ & $\mathcal{O}_{X_9}/k$ & $k_R$ & $\leq2$\\
2 & $E_8\oplus D_7$ & 3 & $T_2,\Theta_5^2$ & $T_2/k_R$ & $k_{R,\tau_9}$ & $\leq4$\\ 
3 & $E_7^{\oplus 2}$ & 3 & $T_1,\Theta_2^2,\Theta_4^1$ & $T_1/k_R$ & $k_{R,\tau_9}$ & $\leq4$\\ 
4 & $E_7\oplus A_8$ & 3 & $\mathcal{O}_{X_9},T_1, \Theta_1^1,\Theta_2^1$ & $\mathcal{O}_{X_9}/k$ & $k_{R,\tau_9}$ & $\leq4$\\ 
5 & $D_{15}$ & 3 & $\Theta_2^2, \Theta_3^2$ & $\Theta_2^2/k_{R,\tau_9}$ & $k_{R,\tau_9}$ & $\leq4$\\ 
6 & $ D_{12}\oplus A_3$ &3 & $\Theta_2^2, \Theta_3^2, \Theta_7^2$  & $\Theta_7^2/k_{R,\tau_9}$ & $k_{R,\tau_9}$ & $\leq4$\\ 
7 & $A_{15}$ & 1 & $\Theta_4^1,\Theta_5^1,\Theta_4^2, \Theta_5^2$ & $\Theta_4^1/k_{R,\tau_9}$ & $k_{R,\tau_9}$ &  $\leq 4$\\
8 & $D_9\oplus A_{6}$ & 3 & $\Theta_4^1,\Theta_2^2,\Theta_6^2,\Theta_7^2$ & $\Theta_4^1/k_{R,\tau_9}$ & $k_{R,\tau_9}$ & $\leq4$\\
9 & $E_6\oplus D_7 \oplus A_2$ & 3 & $\Theta_0^1,\Theta_3^1,\Theta_5^2,\Theta_7^2$ & $\Theta_0^1/k_{\tau_9}$  & $k_{R,\tau_9}$ & $\leq4$\\
10 & $D_6\oplus A_9$ & 1 & $\Theta_7^1, \Theta_7^2,\Theta_2^1$ & $\Theta_7^1/k_{R,\tau_9}$ & $k_{R,\tau_9}$ & $\leq 4$\\
11 & $A_{15}$ &3 &$T_2,\Theta_7^2, \Theta_8^2,\Theta_6^2$&$T_2/k_R$&$k_{R,\tau_9}$&$\leq 4$\\
12 & $A_{12}\oplus A_3$ & 3 & $\mathcal{O}_{X_9},\Theta_2^1,\Theta_7^1,\Theta_4^2,\Theta_5^2$  & $\mathcal{O}_{X_9}/k$ & $k_{R,\tau_9}$ & $\leq 4$ 
\end{tabular}
\caption{Types of the different elliptic fibrations of $X_9$ and fields of definition}
\label{eq: tableof types of the elliptic fibrations on X9} 
\end{table}

\section{The surfaces \texorpdfstring{$R_4$}{R4}, \texorpdfstring{$R_3$}{R3}, \texorpdfstring{$R_2$}{R2} and the surfaces \texorpdfstring{$X_4$}{X4}, \texorpdfstring{$X_3$}{X3}, \texorpdfstring{$X_2$}{X2}}\label{sec: surfaces Ri and Xi}\hspace{0pt} \\

In this section we establish an analogous study for the extremal rational surfaces $R_i$, for $i=4,3,2$. We classify all the possible fibrations of the K3 surfaces $X_i$ and determine their types with respect to the cover involutions $\tau_i$, for $i=4,3,2$.

\subsection{The rational elliptic surfaces \texorpdfstring{$R_4$}{R4}, \texorpdfstring{$R_3$}{R3}, and \texorpdfstring{$R_2$}{R2}}\hspace{0pt}\label{subsec: Ri}\hspace{0pt}

\vspace{5pt}

Let $R_4$ be an extremal rational elliptic surface with one reducible fiber of type $I_4^*$. Its  Mordell--Weil group is $\Z/2\Z=\{\mathcal{O}, t_1\}$, where $\mathcal{O}$ is the zero section and $t_1$ is a 2-torsion section. Recall that a fiber of type $I_4^*$ is given by 9 smooth rational curves meeting with dual graph $\tilde{D_8}$, see \cite[Table I.4.1]{Mi}. The N\'eron--Severi group of $R_4$ contains also the classes of the irreducible components of the reducible fiber, denoted by $C_0, C_1,\ldots, C_8$. The intersections which are not trivial are the following:
$$
\begin{aligned}
&C_l^2=-2, C_0C_2=C_6C_8=1, \\
&C_lC_{j}=1 \quad \text{if and only if} \quad |l-j|=1 \,\text{and}\, \{l,j\}\subset\{2,3,4,5,6\}, \\
&\mathcal{O}C_0=t_1C_8=1,\quad \text{and} \quad \mathcal{O}^2=t_1^2=-1.
\end{aligned}
$$

Let $R_3$ be an extremal rational elliptic surface over $k$ with one reducible fiber of type $III^*$. As $R_3$ is extremal, there is another reducible fiber which is either an $I_2$ or an $III$.  Its Mordell--Weil group is $\Z/2\Z=\{\mathcal{O}, t_1\}$, where $\mathcal{O}$ is the zero section and $t_1$ is a 2-torsion section. Recall that a fiber of type $III^*$ is given by 8 smooth rational curves meeting with dual graph $\tilde{E_7}$, see \cite[Table I.4.1]{Mi}. The N\'eron--Severi group of $R_3$ contains also the classes of the irreducible components of the reducible fiber. Denoted by $C_l$ the components of the $III^*$ fiber and by $D_l$ the ones of the other reducible fiber, the intersections which are not trivial are the following:
$$
\begin{aligned}
&C_l^2=-2, C_lC_{j}=1 \quad \text{if and only if} \quad |l-j|=1 \, \text{and} \, \{l,j\}\subset\{0,1,2,3,4,5,6\},\\
&C_3C_7=1,\quad D_0D_1=2, \quad D_j^2=-2, \\
&\mathcal{O}C_0=t_1C_6=\mathcal{O}D_0=t_1D_1=1 \quad \text{and} \quad \mathcal{O}^2=t_1^2=-1.
\end{aligned}
$$

Let $R_2$ be an extremal rational elliptic surface over $k$ with one reducible fiber of type $II^*$. The other singular fibers are either $II$ or $2I_1$. Its Mordell--Weil group is $\{\mathcal{O}\}$, i.e., it is trivial. Recall that a fiber of type $II^*$ is given by 9 smooth rational curves meeting with dual graph $\tilde{E_8}$, see \cite[Table I.4.1]{Mi}. The N\'eron--Severi group of $R_2$ contains also the classes of the irreducible components of the reducible fiber, denoted by $C_0, C_1,\ldots, C_8$. The intersections which are not trivial are the following:
$$
\begin{aligned}
&C_l^2=-2, C_lC_{j}=1\quad \text{if and only if} \quad |l-j|=1 \, \text{and} \, \{l,j\}\subset\{0,1,2,3,4,5,6,7\},\\
&C_8C_5=1,\quad \mathcal{O}C_0=1, \quad \mathcal{O}^2=-1. \\
\end{aligned}
$$

The following result shows that the surfaces $R_i$ have trivial Galois group $G_{R_i}$, that is its N\'eron-Severi group admits a set of generators over $k$ given by the zero section, a smooth fiber and the non-trivial fiber components of the reducible fibers. It also presents their contractions of negative curves to minimal $k$-rational surfaces.

\begin{proposition}\label{prop: Galois R2, R3 and R4}
Let $R$ be on the following surfaces: $R_2,R_3,R_4$. Then $G_R$ is trivial. 
Moreover, the surfaces $R_2$, $R_3$ and $R_4$ can be contracted to $\mathbb{P}^2$; the surfaces $R_3$ and $R_4$ can be also contracted to $\mathbb{P}^1\times\mathbb{P}^1$ and the surfaces $R_2$ and $R_3$ can be also contracted to $\mathbb{F}_2$, the Hirzebruch surface with a unique $(-2)$-curve.
	\end{proposition}
\proof The proof is similar to the one of Proposition \ref{prop: Galois of R9}. Indeed, for $R=R_2$ or $R_3$, each $g\in G_R$, $g(\mathcal{O})=\mathcal{O}$ and if $\mathrm{MW}=\{\mathcal{O}, t_1\}$, $g(t_1)$ has to be a section different from $\mathcal{O}$ and hence $g(t_1)=t_1$. Thus for each $R_i$, $i=2,3$, the sections are preserved and this implies, arguing via the intersection of the components of the reducible fibers as in Proposition \ref{prop: Galois of R9}, that all the components of the unique reducible fibers are fixed. 

Let us consider the surface of type $R_3$. We have three different possibilities, to obtain three different surfaces:
	\begin{itemize}\item Let us contract the sections $\mathcal{O}$ and $t_1$. Then we contract the images of $C_0$ and $C_6$ (which are now $(-1)$-curves);  the images of $C_1$ and $C_5$; the images of $C_2$
 and $C_4$. There remain the images of $C_3$, which is a curve with self-intersection $0$, and of $C_7$, which is a curve with self-intersection $-2$. There are no $(-1)$-curves on this surface, so we obtain a minimal rational surface, with two independent classes in the N\'eron--Severi group which have self-intersection 0 and $-2$. Hence we obtained $\mathbb{F}_2$
 \item Let us contract first the section $\mathcal{O}$ and then (in this order), the images of the components $C_0$, $C_1$, $C_2$, $C_3$. Now the image of $C_7$ is a $(-1)$-curve. We contract it. It remains a unique $(-1)$-curve, which is the section $t_1$. We contract it and then (in this order) the images of the components $C_6$ and $C_5$. We obtain a minimal rational surface whose N\'eron--Severi group is generated by one class (we contracted 9 curves), which is the image of $C_4$. This rational surface is necessarily $\mathbb{P}^2$.
 \item Let us contract first the section $\mathcal{O}$ and then (in this order), the images of the components $C_0$, $C_1$, $C_2$, $C_3$. Now the image of $C_4$ is a $(-1)$-curve. We contract it. Then we contract $t_1$  and the image of the component $C_5$. We obtain a minimal rational surface, whose N\'eron--Severi group is generated by the two classes which are the images of $C_7$ and $C_5$. Their self-intersection is 0 and they meet in a point, so we obtained $\mathbb{P}^1 \times \mathbb{P}^1$.
\end{itemize}

Let us now consider the surface $R_2$ (see Figure \ref{contraction R_2}). There is a unique $(-1)$-curve, the section $\mathcal{O}$. So we contract it, and than we contract (in this order) the images of the components $C_0$, $C_1$, $C_2$, $C_3$, $C_4$, $C_5$. Now both the images of $C_6$ and $C_8$ are $(-1)$-curves and they meet in a point. \begin{itemize}\item If one contracts the image  of $C_8$, one obtains a minimal surface, whose generators of the N\'eron--Severi group are the images of $C_7$ and $C_6$ and this surface is $\mathbb{F}_2$ (because of the presence of a $(-2)$-curve, image of $C_7$). 
	\item If one contracts the image of $C_6$, then one has to contract the image of $C_7$ and one obtains a minimal rational surface, whose N\'eron--Severi group has one generator (the image of $C_8$) and thus the surface is $\mathbb{P}^2$.
\end{itemize}
\begin{figure}
	\includegraphics[scale=0.1]{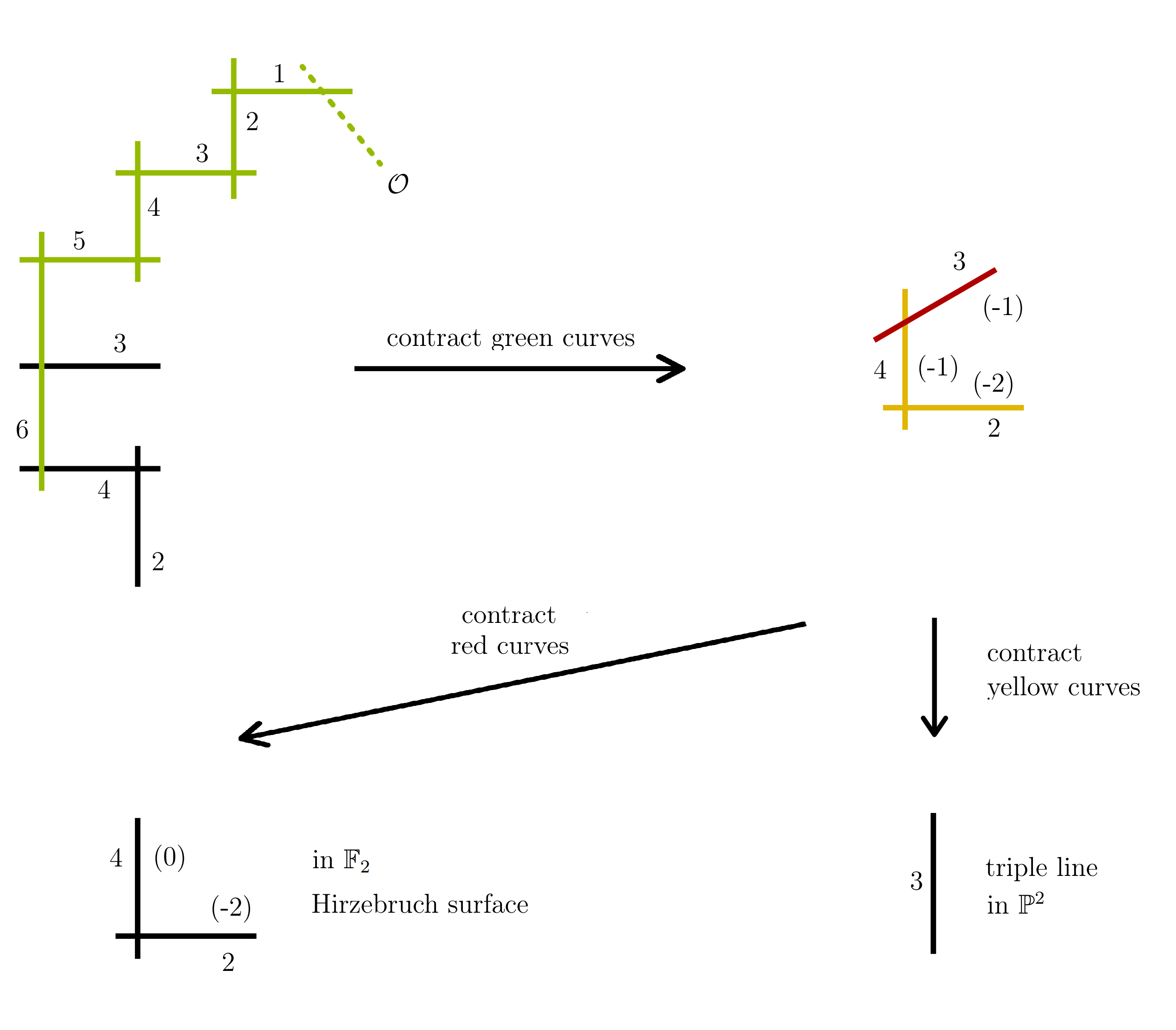}
	\caption{Contractions of $R_2$ to $\mathbb{F}_2$ and to $\mathbb{P}^2$}	\label{contraction R_2}
\end{figure}
Let us consider the surface of type $R_4$. We contract first the section $\mathcal{O}$ and then (in this order) the images of the components $C_0$, $C_2$, $C_3$, $C_4$, $C_5$, $C_6$. Now we have three $(-1)$ curves, i.e. the images of  $C_7$, $C_8$ and $t_1$. The image of $C_8$ meets both the images of $C_7$ and of $t_1$: if one contracts the image of $C_8$, one obtains the minimal surface  $\mathbb{P}^1\times\mathbb{P}^1$; if one contracts the  images of $t_1$ and $C_7$ one obtains $\mathbb{P}^2$.
\endproof

\subsection{The K3 surfaces \texorpdfstring{$X_4$}{X4}, \texorpdfstring{$X_3$}{X3}, \texorpdfstring{$X_2$}{X2}}\hspace{0pt}

\vspace{5pt}

Let $X_i$ be a K3 surface obtained by a generic base change of order 2 on the rational elliptic surface $R_i$ for $i=4,3,2$ as in Section \ref{sec: setting}. Let $P_i$ and $Q_i$ be the points corresponding to the branch fibers of the cover $X_i\rightarrow R_i$. We have the following result, analogous to Proposition \ref{prop: GaloisXi}. 

\begin{proposition}
The Galois group $G_{\mathcal{E}_{X_i}}$ of the elliptic fibration $\mathcal{E}_{X_i}:X_i\ra\mathbb{P}^1$ is contained in $\left(\Z/2\Z\right)$. It is trivial if and only if the points $P_i$ and $Q_i$ are defined over the ground field. 
\end{proposition}
\begin{proof}
The group $G_{R_i}$ is trivial by Proposition \ref{prop: Galois R2, R3 and R4}, so the unique Galois action is the one of the cover involution $\tau_i$, which is trivial if and only if the branch fibers are defined over the ground field.
\end{proof}

The elliptic fibrations $\mathcal{E}_{X_i}$, $i=4,3,2$, are induced by $\mathcal{E}_{R_i}$. We fix the following notation: each component $C_l$ (resp. $D_l$) of a reducible fiber of $\mathcal{E}_{R_i}$ corresponds to two curves $\Theta_l^{j}$ (resp. $\Phi_l^j$), $j=1,2$ on $X_i$ which are components of two different reducible fibers on $X_i$. Moreover the zero section of $\mathcal{E}_{R_i}$ induces the zero section, $\mathcal{O}_{X_i}$, of $\mathcal{E}_{X_i}$ and, if there is a torsion section $t_1$ on $R_i$, it induces a torsion section $T_1$ on $X_i$.

So we have the following curves on $X_i$: $$\Theta_l^j\ j=1,2;\  \mathcal{O}_{X_i};\ T_1\mbox{ if } i\neq 2;\ \Phi_l^j\ j=1,2, l=0,1\mbox{ if }i=3.$$ Denote by $\pi_i:X_i\to R_i$ the double cover of $R_i$ induced by the base change and by $\tau_i$ the cover involution. We have 
$$\pi_i(\mathcal{O}_{X_i})=\mathcal{O},\ \ 
\pi_i(\Theta_l^1)=\pi_i(\Theta_l^2)=C_l,$$ 
$$\pi_i(T_1)=t_1\mbox{ if }i\neq 2,\ \
\pi_i(\Phi_l^1)=\pi_i(\Phi_l^2)=D_l\mbox{ if }i=3.$$
$$\tau_i(\mathcal{O}_{X_i})=\mathcal{O}_{X_i},\  
\tau_i(\Theta_l^1)=\Theta_l^2,\ 
\tau_i(T_1)=T_1\mbox{ if }i\neq 2,\ 
\tau_i(\Phi_l^1)=\Phi_l^2\mbox{ if }i=3.$$

\vspace{5pt}

Figures \ref{diagram X4}, \ref{diagram X3}, and \ref{diagram X2} summarize the above. Note that in Figure \ref{diagram X3}, $\Phi_2^1$ and $\Phi_2^2$ are both connected to $T_1$. 

\begin{center}
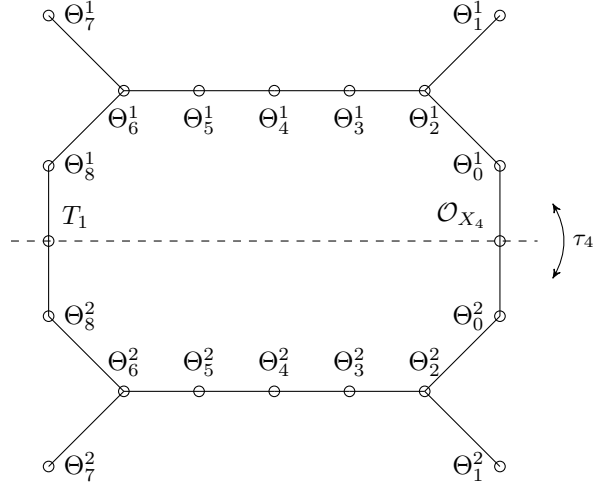
\begin{figure}[h]
\begin{tikzpicture}[scale=1] 
\node [label=above right:$T_1$, draw,circle,inner sep=0pt,minimum size=4pt] at (0,0) {};
\node [label=above left:$\mathcal{O}_{X_4}$, draw,circle,inner sep=0pt,minimum size=4pt] at (6,0) {};
\node [label=left:$\Theta_0^2$,draw,circle,inner sep=0pt,minimum size=4pt] at (6,-1) {};
\node [label=left:$\Theta_1^2$,draw,circle,inner sep=0pt,minimum size=4pt] at (6,-3) {};
\node [label=above:$\Theta_2^2$,draw,circle,inner sep=0pt,minimum size=4pt] at (5,-2) {};
\node [label=above:$\Theta_3^2$,draw,circle,inner sep=0pt,minimum size=4pt] at (4,-2) {};
\node [label=above:$\Theta_4^2$,draw,circle,inner sep=0pt,minimum size=4pt] at (3,-2) {};
\node [label=above:$\Theta_5^2$,draw,circle,inner sep=0pt,minimum size=4pt] at (2,-2) {};
\node [label=above:$\Theta_6^2$,draw,circle,inner sep=0pt,minimum size=4pt] at (1,-2) {};
\node [label=right:$\Theta_7^2$,draw,circle,inner sep=0pt,minimum size=4pt] at (0,-3) {};
\node [label=right:$\Theta_8^2$,draw,circle,inner sep=0pt,minimum size=4pt] at (0,-1) {};
\node [label=left:$\Theta_0^1$,draw,circle,inner sep=0pt,minimum size=4pt] at (6,1) {};
\node [label=left:$\Theta_1^1$,draw,circle,inner sep=0pt,minimum size=4pt] at (6,3) {};
\node [label=below:$\Theta_2^1$,draw,circle,inner sep=0pt,minimum size=4pt] at (5,2) {};
\node [label=below:$\Theta_3^1$,draw,circle,inner sep=0pt,minimum size=4pt] at (4,2) {};
\node [label=below:$\Theta_4^1$,draw,circle,inner sep=0pt,minimum size=4pt] at (3,2) {};
\node [label=below:$\Theta_5^1$,draw,circle,inner sep=0pt,minimum size=4pt] at (2,2) {};
\node [label=below:$\Theta_6^1$,draw,circle,inner sep=0pt,minimum size=4pt] at (1,2) {};
\node [label=right:$\Theta_7^1$,draw,circle,inner sep=0pt,minimum size=4pt] at (0,3) {};
\node [label=right:$\Theta_8^1$,draw,circle,inner sep=0pt,minimum size=4pt] at (0,1) {};

\path[every node/.style={font=\sffamily\small}]
     (0,3) edge (1,2)
     (0,1) edge (1,2)
     (1,2) edge (2,2)
     (2,2) edge (3,2)
     (3,2) edge (4,2)
     (4,2) edge (5,2)
     (5,2) edge (6,3)
     (5,2) edge (6,1)
     (0,1) edge (0,0)
     (0,0) edge (0,-1)
     (0,-1) edge (1,-2)
     (1,-2)	edge (2,-2)
     (2,-2)	edge (3,-2)
     (3,-2)	edge (4,-2)
     (4,-2) edge (5,-2)
     (5,-2) edge (6,-1)
     (5,-2)	edge (6,-3)
     (6,-1)	edge (6,0)
     (6,0)	edge (6,1)
     (0,-3) edge (1,-2) 
     ;
\draw (-0.5,0) -- (6.5,0)[dashed];

\path[<->,>=stealth',auto,every node/.style={font=\sffamily\small}](6.7,-0.5) edge[bend right] node [right]{$\tau_4$} (6.7,0.5);
\end{tikzpicture}
\caption{Reducible fibers and sections of the fibration $\mathcal{E}_{X_4}$ on $X_4$}
\label{diagram X4}
\end{figure} 
\end{center}

\begin{center}
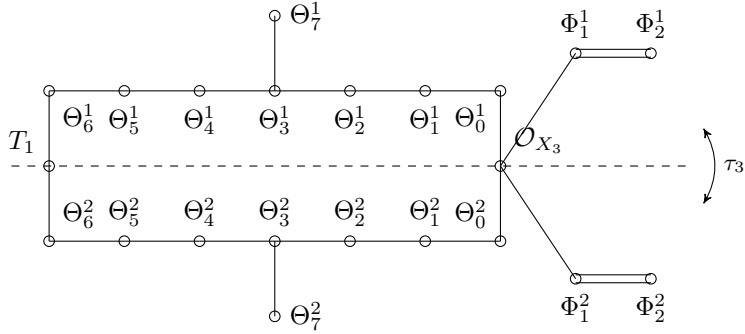
\begin{figure}[h]
\begin{tikzpicture}[scale=1] 
\node [label=above :$\Phi_1^1$, draw,circle,inner sep=0pt,minimum size=4pt] at (7,1.5) {};
\node [label=above :$\Phi_2^1$, draw,circle,inner sep=0pt,minimum size=4pt] at (8,1.5) {};
\node [label=below :$\Phi_1^2$,draw,circle,inner sep=0pt,minimum size=4pt] at (7,-1.5) {};
\node [label=below:$\Phi_2^2$,draw,circle,inner sep=0pt,minimum size=4pt] at (8,-1.5) {};
\node [label=above left:$T_1$, draw,circle,inner sep=0pt,minimum size=4pt] at (0,0) {};
\node [label=above right:$\mathcal{O}_{X_3}$, draw,circle,inner sep=0pt,minimum size=4pt] at (6,0) {};
\node [label=above left:$\Theta_0^2$,draw,circle,inner sep=0pt,minimum size=4pt] at (6,-1) {};
\node [label=above:$\Theta_1^2$,draw,circle,inner sep=0pt,minimum size=4pt] at (5,-1) {};
\node [label=above:$\Theta_2^2$,draw,circle,inner sep=0pt,minimum size=4pt] at (4,-1) {};
\node [label=above:$\Theta_3^2$,draw,circle,inner sep=0pt,minimum size=4pt] at (3,-1) {};
\node [label=above:$\Theta_4^2$,draw,circle,inner sep=0pt,minimum size=4pt] at (2,-1) {};
\node [label=above:$\Theta_5^2$,draw,circle,inner sep=0pt,minimum size=4pt] at (1,-1) {};
\node [label=above right:$\Theta_6^2$,draw,circle,inner sep=0pt,minimum size=4pt] at (0,-1) {};
\node [label=right:$\Theta_7^2$,draw,circle,inner sep=0pt,minimum size=4pt] at (3,-2) {};
\node [label=below left:$\Theta_0^1$,draw,circle,inner sep=0pt,minimum size=4pt] at (6,1) {};
\node [label=below:$\Theta_1^1$,draw,circle,inner sep=0pt,minimum size=4pt] at (5,1) {};
\node [label=below:$\Theta_2^1$,draw,circle,inner sep=0pt,minimum size=4pt] at (4,1) {};
\node [label=below:$\Theta_3^1$,draw,circle,inner sep=0pt,minimum size=4pt] at (3,1) {};
\node [label=below:$\Theta_4^1$,draw,circle,inner sep=0pt,minimum size=4pt] at (2,1) {};
\node [label=below:$\Theta_5^1$,draw,circle,inner sep=0pt,minimum size=4pt] at (1,1) {};
\node [label=below right:$\Theta_6^1$,draw,circle,inner sep=0pt,minimum size=4pt] at (0,1) {};
\node [label=right:$\Theta_7^1$,draw,circle,inner sep=0pt,minimum size=4pt] at (3,2) {};

\path[every node/.style={font=\sffamily\small}]
     (0,1) edge (1,1)
     (1,1) edge (2,1)
     (2,1) edge (3,1)
     (3,1) edge (4,1)
     (4,1) edge (5,1)
     (5,1) edge (6,1)
     (3,1) edge (3,2)
     (0,1) edge (0,0)
     (0,0) edge (0,-1)
     (6,0)  edge (6,1)
     (0,-1) edge (1,-1)
     (1,-1)	edge (2,-1)
     (2,-1)	edge (3,-1)
     (3,-1)	edge (4,-1)
     (4,-1) edge (5,-1)
     (5,-1) edge (6,-1)
     (6,-1) edge (6,0)
     (3,-1)	edge (3,-2) 
     (6,0)  edge (7,1.5)
     (6,0)  edge (7,-1.5)
     ;
\draw (-0.5,0) -- (8.5,0)[dashed];
\draw[transform canvas={yshift=-1.5pt}] (7,1.5) -- (8,1.5);
\draw[transform canvas={yshift=1.5pt}] (8,1.5) -- (7,1.5);
\draw[transform canvas={yshift=-1.5pt}] (7,-1.5) -- (8,-1.5);
\draw[transform canvas={yshift=1.5pt}] (8,-1.5) -- (7,-1.5);
\path[<->,>=stealth',auto,every node/.style={font=\sffamily\small}](8.7,-0.5) edge[bend right] node [right]{$\tau_3$} (8.7,0.5);
\end{tikzpicture}
\caption{Reducible fibers and sections of the fibration $\mathcal{E}_{X_3}$ on $X_3$}
\label{diagram X3}
\end{figure} 
\end{center}

\begin{center}
\begin{figure}[h]
\begin{tikzpicture}[scale=1] 
\node [label=above right:$\mathcal{O}_{X_2}$, draw,circle,inner sep=0pt,minimum size=4pt] at (6,0) {};
\node [label=above left:$\Theta_0^2$,draw,circle,inner sep=0pt,minimum size=4pt] at (6,-1) {};
\node [label=above:$\Theta_1^2$,draw,circle,inner sep=0pt,minimum size=4pt] at (5,-1) {};
\node [label=above:$\Theta_2^2$,draw,circle,inner sep=0pt,minimum size=4pt] at (4,-1) {};
\node [label=above:$\Theta_3^2$,draw,circle,inner sep=0pt,minimum size=4pt] at (3,-1) {};
\node [label=above:$\Theta_4^2$,draw,circle,inner sep=0pt,minimum size=4pt] at (2,-1) {};
\node [label=above:$\Theta_5^2$,draw,circle,inner sep=0pt,minimum size=4pt] at (1,-1) {};
\node [label=above:$\Theta_6^2$,draw,circle,inner sep=0pt,minimum size=4pt] at (0,-1) {};
\node [label=above:$\Theta_7^2$,draw,circle,inner sep=0pt,minimum size=4pt] at (-1,-1) {};
\node [label=right:$\Theta_8^2$,draw,circle,inner sep=0pt,minimum size=4pt] at (1,-2) {};
\node [label=below left:$\Theta_0^1$,draw,circle,inner sep=0pt,minimum size=4pt] at (6,1) {};
\node [label=below:$\Theta_1^1$,draw,circle,inner sep=0pt,minimum size=4pt] at (5,1) {};
\node [label=below:$\Theta_2^1$,draw,circle,inner sep=0pt,minimum size=4pt] at (4,1) {};
\node [label=below:$\Theta_3^1$,draw,circle,inner sep=0pt,minimum size=4pt] at (3,1) {};
\node [label=below:$\Theta_4^1$,draw,circle,inner sep=0pt,minimum size=4pt] at (2,1) {};
\node [label=below:$\Theta_5^1$,draw,circle,inner sep=0pt,minimum size=4pt] at (1,1) {};
\node [label=below :$\Theta_6^1$,draw,circle,inner sep=0pt,minimum size=4pt] at (0,1) {};
\node [label=below:$\Theta_7^1$,draw,circle,inner sep=0pt,minimum size=4pt] at (-1,1) {};
\node [label=right:$\Theta_8^1$,draw,circle,inner sep=0pt,minimum size=4pt] at (1,2) {};

\path[every node/.style={font=\sffamily\small}]
     (0,1) edge (1,1)
     (1,1) edge (2,1)
     (2,1) edge (3,1)
     (3,1) edge (4,1)
     (4,1) edge (5,1)
     (5,1) edge (6,1)
     (1,1) edge (1,2)
     (6,0)  edge (6,1)
     (0,-1) edge (1,-1)
     (1,-1)	edge (2,-1)
     (2,-1)	edge (3,-1)
     (3,-1)	edge (4,-1)
     (4,-1) edge (5,-1)
     (5,-1) edge (6,-1)
     (6,-1) edge (6,0)
     (1,-1)	edge (1,-2)
     (0,1) edge (-1,1)
     (0,-1) edge (-1,-1) 
     ;
\draw (-1.5,0) -- (6.5,0)[dashed];
\path[<->,>=stealth',auto,every node/.style={font=\sffamily\small}](6.7,-0.5) edge[bend right] node [right]{$\tau_2$} (6.7,0.5);
\end{tikzpicture}
\caption{Reducible fibers and sections of the fibration $\mathcal{E}_{X_2}$ on $X_2$}
\label{diagram X2}
\end{figure}
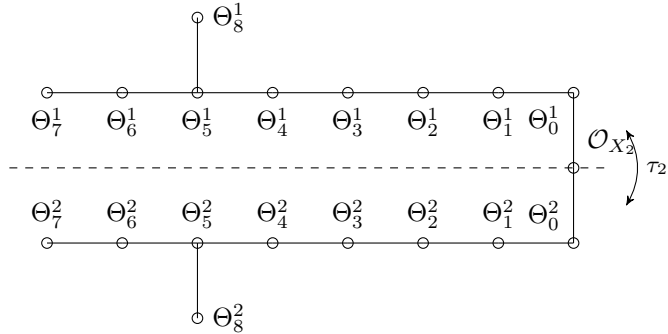 
\end{center}

\begin{proposition}\label{T for X2 X3 X4} The N\'eron--Severi group of $X_i$ has rank 18, signature $(1,17)$, for every $i=2,3,4$. 

Both lattices $\mathrm{NS}(X_4)$ and of $\mathrm{NS}(X_3)$ are isometric to $U\oplus D_8\oplus E_8$ and their transcendental lattices are both isometric to $U\oplus U(2)$, which has the same  discriminant group and form as $D_8$. In particular $X_3$ and $X_4$ lie in the same family of K3 surfaces, namely the family of $U\oplus D_8\oplus E_8$-polarized K3 surfaces.

The lattice $\mathrm{NS}(X_2)$ is isometric to $U\oplus E_8\oplus E_8$ and its transcendental lattice is isometric to $U\oplus U$, which has the same discriminant form of $E_8$.
\end{proposition}
\begin{proof}
The curves in the Figures \ref{diagram X4}, \ref{diagram X3}, and \ref{diagram X2} (i.e. the curves $\Theta_l^j$, $\mathcal{O}_{X_i}$, $T_1$ if $i\neq 2$ and $\Phi_l^j$ if $i=3$) generate $\mathrm{NS}(X_i)$. They are not all linearly independent, but once one extract a basis, one obtains 18 independent generators of $\mathrm{NS}(X_i)$. Since one knows all the intersection properties of these generators, one can explicitly compute their intersection matrix. This identifies the lattice $\mathrm{NS}(X_{i})$ and in particular its discriminant group and form. We observe that all the lattices that appear are 2-elementary, i.e., the discriminant group is $\left(\Z/2\Z\right)^a$, $a\in\N$. So the transcendental lattice is a 2-elementary lattice with signature $(2,2)$. The indefinite 2-elementary lattices are completely determined by their length, i.e., by $a$, and by another invariant, often denoted by $\delta$, which is zero in all the cases considered. This allows us to identify the transcendental lattices. 
\end{proof}

\subsection{Classification of all the possible fibrations on the \texorpdfstring{$K3$}{K3} surfaces \texorpdfstring{$X_4$}{X4}, \texorpdfstring{$X_3$}{X3}, and \texorpdfstring{$X_2$}{X2}}\label{sec: classification X2 X3 X4}\hspace{0pt}

\vspace{5pt}

In the same way as we did for $X_9$ in Section \ref{sec: classification of the fibrations of X9}, we classify elliptic fibrations on the surfaces $X_4\simeq X_3$ and $X_2$  in what follows. 
By Proposition \ref{T for X2 X3 X4} we take $T=D_8$ for $X_3\simeq X_4$, and $T=E_8$ for $X_2$ and apply Nishiyama's method. 
By \cite[Lemmas 4.1 and 4.3]{Nish} we know that $D_8$ only embeds primitively in $D_n$ for $n\geq8$, and $E_8$ only embeds primitively in $E_8$. The orthogonal complements of these embeddings in the 24 Niemeier lattices are then found in \cite[Corollary 4.4]{Nish}.
Those results are summarized in tables \ref{eq: table elliptic fibrations on X3 X4} and \ref{eq: table elliptic fibrations on X2}. 
We notice that the fibrations on $X_2$, $X_3$ and $X_4$ were already classified in \cite[Table 2, case $k=8$ and Table 1 case $k=8,\delta=0$]{GS2} via different methods.

\begin{table}[h!]
	\begin{adjustwidth}{+0cm}{}		
		\begin{tabular}{c|c|c|c|c|c}
			$n^o$&Niemeier&embedding&roots orth.&reducible fibers&MW\\
			1& $E_8\oplus D_{16}$&$D_8 \subset D_{16}$&  $E_8 \oplus  D_8$ &$II^*+ I_{4}^*$& $\{\mathcal{O}\}$\\
			2&$E_7^{\oplus2}\oplus D_{10}$&$D_8 \subset D_{10}$&$E_7^{\oplus2}\oplus A_1^{\oplus2}$& $2III^*+2I_2$&$\mathbb{Z}/2 \mathbb{Z}$\\
			3&$D_{24}$ &$D_8 \subset D_{24}$&$D_{16}$&$I_{12}^*$&$\{\mathcal{O}\}$\\
			4&$D_{12}^{\oplus2}$&$D_8 \subset D_{12}$&$D_{12}\oplus D_4$&$I_{8}^*+I_0^*$&$\mathbb{Z}/2 \mathbb{Z}$\\
			5&$D_{8}^{\oplus3}$&$D_8 \subset D_{8}$&$  D_8^{\oplus2}$&$2I_4^*$&$\mathbb{Z}/2 \mathbb{Z}$\\
			6&$D_{9} \oplus A_{15}$&$D_8 \subset D_{9}$&$A_{15}$&$I_{16}$&$\mathbb{Z}/2 \mathbb{Z}\oplus \mathbb{Z}$ 
		\end{tabular}
		\caption{Elliptic fibrations of $X_3$ and $X_4$}
		\label{eq: table elliptic fibrations on X3 X4} 
	\end{adjustwidth}
\end{table}

\begin{table}[h!]
	\begin{adjustwidth}{+0cm}{}
		\begin{tabular}{c|c|c|c|c|c}
		    $n^o$&Niemeier&embedding&roots orth.&reducible fibers&MW\\
			1& $E_8^{\oplus3}$ & $E_8\subset E_8$ & $E_8^{\oplus2}$ & $2II^*$ & $\{\mathcal{O}\}$\\
			2& $E_8\oplus D_{16}$&$E_8 \subset E_{8}$&  $D_{16}$ & $I_{12}^*$& $\Z/2\Z$ 
		\end{tabular}
		\caption{Elliptic fibrations of $X_2$}
		\label{eq: table elliptic fibrations on X2}
\end{adjustwidth}
\end{table}

\subsection{Determining the type of each fibration of \texorpdfstring{$X_4$}{X4}, \texorpdfstring{$X_3$}{X3}, and \texorpdfstring{$X_2$}{X2}}\hspace{0pt}

\vspace{5pt}

As in Section \ref{sec:type fibrations X9} we  determine the type of each fibration obtained in Section \ref{sec: classification X2 X3 X4} (Tables \ref{eq: table elliptic fibrations on X3 X4} and \ref{eq: table elliptic fibrations on X2}) with respect to the cover involutions $\tau_i$, for $i=4,3,2$. We determine moreover the sections and their fields of definition. This study allows us to obtain an upper bound for the degree over $k$ of a field of definition $k_{\eta}$ of a given fibration $\eta$, and an upper bound for the degree over $k$ of a field of definition $k_{\eta, \mathrm{MW}}$ of a set of generators of the Mordell--Weil group of the fibration.

By Proposition \ref{prop: Galois R2, R3 and R4} we know that the Galois group $G_{R_i}$ is trivial and all the fiber components of $R_i$ are defined over $k$, for $i=4,3,2$. 
In order to determine the field of definition of the sections, the only action that is taken into account is the one of the cover involutions $\tau_i$, for $i=4,3,2$.

To determine the type of each fibration in Table \ref{eq: table elliptic fibrations on X3 X4} (resp. Table \ref{eq: table elliptic fibrations on X2}) with respect to $\tau_4$ (resp. $\tau_3$ and $\tau_2$), we find a configuration of (parts of the) reducible fibers in terms of the curves in Figure~\ref{diagram X4} (resp. Figure \ref{diagram X3} and Figure \ref{diagram X2}). The fibration in line 5 (resp. line 2 and line 1) is represented in Figure \ref{diagram X4} (resp. Figure \ref{diagram X3} and Figure \ref{diagram X2}). 

The configurations associated to the fibers in lines 1, 2, 3, 4 and 6 in Table \ref{eq: table elliptic fibrations on X3 X4} for the K3 surface $X_4$ are listed below:
\vspace{5pt}

\noindent$\underline{II^*+I_4^*}$\\
$\Theta_0^1,2\Theta_2^1,3\Theta_3^1,4\Theta_4^1,5\Theta_5^1,6\Theta_6^1,4\Theta_8^1,2T_1,3\Theta_7^1\;\;+\;\;\Theta_0^2,\Theta_1^2,2\Theta_2^2,2\Theta_3^2,2\Theta_4^2,2\Theta_5^2,2\Theta_6^2,\Theta_7^2,$

\vspace{5pt}

\noindent$\underline{2III^*+2I_2}$\\
$\Theta_3^1,2\Theta_4^1,3\Theta_5^1,4\Theta_6^1,3\Theta_8^1,2T_1,\Theta_8^2,2\Theta_7^1\;\;+\;\;\Theta_0^1,2\mathcal{O}_{X_4},3\Theta_0^2,4\Theta_2^2,3\Theta_3^2,2\Theta_4^2,\Theta_5^2,2\Theta_1^2\;\;+\;\;\Theta_1^1\;\;+\;\;\Theta_7^2,$

\vspace{5pt}
\noindent$\underline{I_{12}^*}$\\
$\Theta_7^1,\Theta_8^1,2\Theta_6^1,2\Theta_5^1,2\Theta_4^1,2\Theta_3^1,2\Theta_2^1,2\Theta_0^1,2\mathcal{O}_{X_4},2\Theta_0^2,2\Theta_2^2,2\Theta_3^2,2\Theta_4^2,2\Theta_5^2,2\Theta_6^2,\Theta_8^2,\Theta_7^2,$

\vspace{5pt}
\noindent$\underline{I_8^*+I_0^*}$ \\$\Theta_8^1,\Theta_7^1,2\Theta_6^1,2\Theta_5^1,2\Theta_4^1,2\Theta_3^1,2
\Theta_2^1,2\Theta_0^1,2\mathcal{O}_{X_4},2\Theta_0^2,2\Theta_2^2,\Theta_1^2,\Theta_3^1\;\;+\;\;\Theta_8^2,\Theta_7^2,2\Theta_6^2,\Theta_5^2,$

\vspace{5pt}
\noindent$\underline{I_{16}}$ \\$\Theta_8^1,\Theta_6^1,\Theta_5^1,\Theta_4^1,\Theta_3^1,
\Theta_2^1,\Theta_0^1,\mathcal{O}_{X_4},\Theta_0^2,\Theta_2^2,\Theta_3^2,\Theta_4^2,\Theta_5^2,\Theta_6^2,\Theta_8^2, T_1.$
\vspace{5pt}

The configurations associated to the fibers in lines 1, 3, 4, 5 and 6 in Table \ref{eq: table elliptic fibrations on X3 X4} for the K3 surface $X_3$ are listed below:

\vspace{5pt}
\noindent$\underline{II^*+I_4^*}$\\
$\Theta_0^2,2\mathcal{O}_{X_3},3\Theta_0^1,4\Theta_1^1,5\Theta_2^1,6\Theta_3^1,4\Theta_4^1,
2\Theta_5^1,3\Theta_7^1$

\hspace{120pt} $+\;\;\Phi_2^1,\Phi_2^2,2T_1,2\Theta_6^2,2\Theta_5^2,2\Theta_4^2,2\Theta_3^2,\Theta_7^2,\Theta_2^2,$

\vspace{5pt}

\noindent$\underline{I_{12}^*}$\\
$\Phi_1^1,\Phi_1^2,2\mathcal{O}_{X_3},2\Theta_0^1,2\Theta_1^1,2\Theta_2^1,2\Theta_3^1,2\Theta_4^1,2\Theta_5^1,2\Theta_6^1,2T_1,2\Theta_6^2,
2\Theta_5^2,2\Theta_4^2,2\Theta_3^2,\Theta_2^2,\Theta_7^2,$

\vspace{5pt}

\noindent$\underline{I_8^*+I_0^*}$\\
$\Theta_2^1,\Theta_7^1,2\Theta_3^1,2\Theta_4^1,2\Theta_5^1,2\Theta_6^1,2
T_1,2\Theta_6^2,2\Theta_5^2,2\Theta_4^2,2\Theta_3^2,\Theta_2^2,\Theta_7^2\;+\;\Theta_0^2,\Theta_0^1,2\mathcal{O}_{X_3},\Phi_1^1,\Phi_1^2,$ 

\vspace{5pt}

\noindent$\underline{2I_4^*}$\\
$\Phi_1^1,\Phi_1^2,2\mathcal{O}_{X_3},2\Theta_0^1,2\Theta_1^1,2\Theta_2^1,2\Theta_3^1,\Theta_4^1,\Theta_7^1\;\;+\;\;\Theta_2^2,\Theta_7^2,2\Theta_3^2,2\Theta_4^2,2\Theta_5^2,2\Theta_6^2,2T_1,\Theta_6^1,$

\vspace{5pt}

\noindent$\underline{I_{16}}$\\
$\Theta_3^1,\Theta_4^1,\Theta_5^1,\Theta_6^1,
T_1,\Theta_6^2,\Theta_5^2,\Theta_4^2,\Theta_3^2,\Theta_2^2,\Theta_1^2, \Theta_0^2,\mathcal{O}_{X_3},\Theta_0^1,\Theta_1^1, \Theta_2^1.$ 

\vspace{5pt}

Finally, in terms of the configuration of $2II^*$ we find a fiber of type $I_{12}^*$ representing the fibration in line 2 of Table \ref{eq: table elliptic fibrations on X2} by including every curve in Figure \ref{diagram X2} except $\Theta_7^1,\Theta_7^2$; The latter are sections for this fibration. 

\vspace{11pt}

Note that all the reducible fibers listed above only appear once in Table \ref{eq: table elliptic fibrations on X3 X4} (resp. Table \ref{eq: table elliptic fibrations on X2}), hence we know that they represent the corresponding fibrations in those tables. Therefore, using these configurations, we can determine the type of the corresponding fibration with respect to $\tau_4$ (resp. $\tau_3$ and $\tau_2$), and find sections for the corresponding fibration. By choosing a 0-section, we determine whether the different sections are fixed by $\tau_4$ (resp. $\tau_3$ and $\tau_2$) or not. The results are listed in Table \ref{eq: table types elliptic fibrations on X4} (resp. Table \ref{eq: table types elliptic fibrations on X3} and Table \ref{eq: table types elliptic fibrations on X2}). 

\begin{table}[h!]
		\begin{tabular}{c|c|c|c|c|c|c}
$\mbox{n}^o$&roots orth.&type&sections&\makecell{field of \\def. \\0-section}&\makecell{field of\\def. all\\sections}&$[k_{\eta, \mathrm{MW}}:k]$\\
			1& $E_8\oplus D_8$ & 3 & $\mathcal{O}_{X_4}$ & $\mathcal{O}_{X_4}/k$ & $k$ & 1\\
			2 & $E_7^{\oplus 2} \oplus A_1^{\oplus 2}$ & 3 & $\Theta_2^1,\Theta_6^2$ & $\Theta_2^1/k_{\tau_4}$ & $k_{\tau_4}$ & $\leq2$ \\ 
			3 & $D_{16}$ & 1 & $T_1$ & $T_1/k$ & $k$ & 1 \\
			4 & $D_{12} \oplus D_4$ & 3 & $T_1,\Theta_4^2$ & $T_1/k$ & $k_{\tau_4}$ & $\leq2$ \\
			5 & $D_8^{\oplus 2}$ & 2 & $\mathcal{O}_{X_4},T_1$ & $\mathcal{O}_{X_4}/k$ & $k$ & 1 \\
			6 & $A_{15}$ & 1 & $\Theta_1^1,\Theta_7^1, \Theta_1^2$ & $\Theta_1^1/k_{\tau_4}$ & $k_{\tau_4}$& $\leq 2$ \\
		\end{tabular}
		\caption{Types of the different elliptic fibrations of $X_4$ with respect to $\tau_4$ and fields of definition}
		\label{eq: table types elliptic fibrations on X4}
\end{table}

\begin{table}[h!]
		\begin{tabular}{c|c|c|c|c|c|c}
			$\mbox{n}^o$&roots orth.&type&sections&\makecell{field of\\ def. \\0-section}&\makecell{field of\\def. all\\sections}&$[k_{\eta, \mathrm{MW}}:k]$\\
			1& $E_8\oplus D_8$ & 3 & $\Theta_1^2$ & $\Theta_1^2/k_{\tau_3}$ & $k_{\tau_3}$ & $\leq 2$\\
			2 & $E_7^{\oplus 2} \oplus A_1^{\oplus 2}$ & 2 & $\mathcal{O}_{X_3}, T_1$ & $\mathcal{O}_{X_3}/k$&$k$&1\\ 
			3 & $D_{16}$ & 3 & $\Theta_1^2$ & $\Theta_1^2/k_{\tau_3}$ &$k_{\tau_3}$& $\leq 2$ \\
			4 & $D_{12} \oplus D_4$ & 1 & $\Theta_1^1, \Theta_1^2$ & $\Theta_1^2/k_{\tau_3}$ & $k_{\tau_3}$ & $\leq2$ \\
			5 & $D_8^{\oplus 2}$ & 3 & $\Theta_5^1, \Theta_1^2$ & $\Theta_5^1/k_{\tau_3}$ & $k_{\tau_3}$&$\leq 2$ \\
			6 & $A_{15}$ & 1 & $\Theta_7^1, \Theta_7^2, \Phi_1^1, $ & $\Theta_7^1/k_{\tau_3}$ & $k_{\tau_3}$& $\leq2$ 
		\end{tabular}
		\caption{Types of the different elliptic fibrations of $X_3$ with respect to $\tau_3$ and fields of definition}
		\label{eq: table types elliptic fibrations on X3}
\end{table}

\begin{table}[h!]
		\begin{tabular}{c|c|c|c|c|c|c}
			$\mbox{n}^o$&roots orth.&type&sections&\makecell{field of\\ def. \\0-section}&\makecell{field of\\def. all\\sections}&$[k_{\eta, \mathrm{MW}}:k]$\\
			1& $E_8^{\oplus 2}$ &2 & $\mathcal{O}$ & $\mathcal{O}/k$ & $k$ & 1\\
			2 & $D_{16}$ & 1& $\Theta_7^1,\Theta_7^2$ & $\Theta_7^1/k_{\tau_2}$ & $k_{\tau_2}$ & $\leq2$    \\ 
		\end{tabular}
		\caption{Types of the different elliptic fibrations of $X_2$ with respect to $\tau_2$ and fields of definition}
		\label{eq: table types elliptic fibrations on X2}
\end{table}

\vspace{1cm}

\paragraph{\bf Acknowledgements}  The authors would like to thank the organizing committee of the third \textit{Women in Numbers Europe} workshop, where this project started. They would also like to express gratitude to the referees for their careful reading and suggestions that improved the paper.
Victoria Cantoral-Farf\'an was partially supported by KU Leuven IF C14/17/083.
 Cec\'ilia Salgado was partially supported by FAPERJ grant E-26/203.205/2016,  the Serrapilheira Institute (grant Serra-1709-17759), Cnpq grant PQ2 310070/2017-1 and the Capes-Humboldt program.
Antolena Trbovi\'c was supported by the QuantiXLie Centre of Excellence, a project co-financed by the Croatian
Government and European Union through the European Regional Development Fund - the Competitiveness and
Cohesion Operational Programme (Grant KK.01.1.1.01.0004).

\end{document}